\documentclass[11pt,a4paper]{amsart}

\usepackage[utf8]{inputenc}
\usepackage[T1]{fontenc}
\usepackage{amsmath,amssymb,amsthm}
\usepackage[margin=1.35in]{geometry}
\usepackage{xcolor}
\usepackage{hyperref}

\newtheorem{theorem}{Theorem}[section]
\newtheorem{lemma}[theorem]{Lemma}
\newtheorem{proposition}[theorem]{Proposition}
\newtheorem{corollary}[theorem]{Corollary}
\newtheorem{conjecture}[theorem]{Conjecture}
\newtheorem*{thmA}{Theorem A}

\newtheorem{innercustomlemma}{Lemma}
\newenvironment{customlemma}[1]
  {\renewcommand\theinnercustomlemma{#1}\innercustomlemma}
  {\endinnercustomlemma}

\theoremstyle{definition}
\newtheorem{definition}[theorem]{Definition}
\theoremstyle{remark}
\newtheorem{remark}[theorem]{Remark}

\numberwithin{equation}{section}

\DeclareMathOperator{\divg}{div}
\DeclareMathOperator{\Hess}{Hess}
\DeclareMathOperator{\tr}{tr}
\DeclareMathOperator{\supp}{supp}

\DeclareMathOperator{\Ric}{Ric}
\DeclareMathOperator{\Cut}{Cut}
\newcommand{\R}{\mathbb{R}}
\newcommand{\Hy}{\mathbb{H}}
\newcommand{\gl}{\nabla_{L}}
\newcommand{\lapL}{\Delta_{L}}
\newcommand{\ct}{\mathrm{ct}}
\newcommand{\cB}{\mathcal B}

\title[Mean curvature bounds for CMC foliations]{Mean curvature bounds for CMC foliations of manifolds with sectional curvature bounded below}

\author[J. E. Sampaio]{Jos\'e Edson Sampaio}
\author[E. C. da Silva]{Euripedes Carvalho da Silva}

\address{Jos\'e Edson Sampaio:  
              Departamento de Matem\'atica, Universidade Federal do Cear\'a,
	      Rua Campus do Pici, s/n, Bloco 914, Pici, 60440-900, 
	      Fortaleza-CE, Brazil. \newline   
              E-mail: {\tt edsonsampaio@mat.ufc.br}                    
}
\address{Euripedes Carvalho da Silva: Departamento de Matem\'atica, Instituto Federal de Educa\c{c}\~ao, Ci\^encia e Tecnologia do Cear\'a,
 	      Av. Parque Central, 1315, Distrito Industrial I, 61939-140, 
 	      Maracana\'u-CE, Brazil. \newline 
 	      and 	      Departamento de Matem\'atica, Universidade Federal do Cear\'a,
	      Rua Campus do Pici, s/n, Bloco 914, Pici, 60440-900, 
	      Fortaleza-CE, Brazil.
               E-mail: {\tt euripedes.carvalho@ifce.edu.br}
 } 

\thanks{The first named author was partially supported by CNPq-Brazil grants 303375/2025-6, 404611/2025-7 and 405618/2023-9. The second named author was partially supported by CNPq-Brazil grants 400043/2025-4 and 405618/2023-9. This work was supported by the Serrapilheira Institute (grant number Serra -- R-2110-39576).
}
\begin{document}
\emergencystretch=3em

\begin{abstract}
Let $M^{n+1}$ be a complete Riemannian manifold with $\sec_M\ge-K_0$, $K_0\ge0$, and let $\mathfrak{F}$ be a codimension one foliation of $M$ by hypersurfaces of constant mean curvature, the constant depending on the leaf. We prove that the mean curvature of every leaf satisfies $|H|\le(n+1)\sqrt{K_0}$. In particular, the mean curvature of the leaves of a CMC foliation of a complete manifold with $|\sec_M|\le1$ is bounded by $n+1$, which answers positively a conjecture of Meeks III, P\'erez, and Ros.
\end{abstract}

\maketitle
\tableofcontents

\section{Introduction}\label{sec:intro}

Let $M^{n+1}$, $n\ge1$, be a complete connected Riemannian manifold, and let $\mathfrak{F}$ be a smooth codimension one foliation of $M$ by connected hypersurfaces. We call $\mathfrak{F}$ a \emph{CMC foliation} if every leaf has constant mean curvature, the constant possibly depending on the leaf. Every codimension one foliation becomes transversely orientable after passing to a covering of $M$ of degree at most two; the lifted foliation is again a CMC foliation whose leaves cover those of $\mathfrak{F}$ and have the same mean curvature up to sign, and completeness and lower sectional curvature bounds are inherited by Riemannian coverings. Since only $|H|$ will be estimated, we assume throughout that $\mathfrak{F}$ is transversely oriented. Let $N$ be a smooth unit vector field normal to the leaves and define $H\in C^\infty(M)$ by
\begin{equation}\label{eq:H}
\divg N=-nH .
\end{equation}
With the shape operator $A:=-\nabla N$, the restriction of $A$ to the tangent space of a leaf has trace $nH$ (see \S\ref{sec:prelim}), so $H$ is the mean curvature of the leaves with respect to $N$.

In 2008, Meeks III, P\'erez, and Ros \cite[Remark 5.21]{MeeksPR:2008} proposed the following conjecture.
\begin{conjecture}\label{conjecture:control_mean_curv_weaker}
For each integer $n\ge1$ there is $A=A(n+1)>0$ with the following property: for every complete Riemannian manifold $M^{n+1}$ with absolute sectional curvature bounded from above by $1$ and every codimension one CMC foliation $\mathfrak{F}$ of $M$, the mean curvature $H_L$ of each leaf $L$ of $\mathfrak{F}$ satisfies $|H_L|\le A$.
\end{conjecture}
They also conjectured that one can take $A(n+1)=1$ \cite[Conjecture 5.1(2)]{MeeksPR:2008}:
\begin{conjecture}\label{conjecture:control_mean_curv}
Let $\mathfrak{F}$ be a codimension one CMC foliation of a complete Riemannian manifold $M^{n+1}$ with absolute sectional curvature bounded from above by $1$. Then $|H_L|\le1$ for every leaf $L$ of $\mathfrak{F}$.
\end{conjecture}
Meeks III, P\'erez, and Ros proved Conjecture~\ref{conjecture:control_mean_curv} when $M$ is the hyperbolic space $\Hy^3$ \cite[Corollary 5.10]{MeeksPR:2008}, and they provided in \cite[Theorem 5.23]{MeeksPR:2008} partial positive answers to Conjecture \ref{conjecture:control_mean_curv_weaker}, they proved that when $M$ is a homogeneously regular manifold with absolute sectional curvature at most $1$ and $n=3$ or $4$, the mean curvature of the leaves of a CMC foliation of $M$ is bounded by a constant depending only on $n$. When $\dim M=3$, a bound for $|H|$ depending only on an upper bound for $|\sec_M|$ can also be derived from the universal curvature estimate for CMC foliations of compact Riemannian $3$-manifolds with boundary \cite[Theorem 1.3]{MPR16}: applied to a compact domain of $M$ containing a ball of radius $\pi$ about a point $p$, it bounds the norm of the second fundamental form of the leaves at $p$, hence their mean curvature, by a universal constant. When $M$ is compact, the sharp bound $|H|\le\sqrt{K_0}$ is known even under the weaker assumption $\Ric_M\ge-nK_0$; see \cite[Theorem 3.1]{BKO} for $K_0=0$, \cite[Theorem 5.8.C]{MeeksPR:2008}, \cite[Remark 1.4(i)]{MP16}, and \cite[Theorem 4.4]{SS}. Recently, the authors proved in \cite{SS2026} that if $\sec_M\equiv 0$, then $H\equiv 0$. 

The aim of this article is to provide a complete positive answer to Conjecture \ref{conjecture:control_mean_curv_weaker}. Indeed, we obtain the following more general result.

\begin{thmA}
Let $M^{n+1}$ be a complete Riemannian manifold with $\sec_M\ge-K_0$ for some $K_0\ge0$, and let $\mathfrak{F}$ be a codimension one CMC foliation of $M$. Then
\[
|H|\le(n+1)\sqrt{K_0}\qquad\text{on }M .
\]
In particular, if $\sec_M\ge0$, every leaf of $\mathfrak{F}$ is minimal.
\end{thmA}

\begin{corollary}\label{cor:conj}
Conjecture~\ref{conjecture:control_mean_curv_weaker} holds with $A(n+1)=n+1$. More generally, if $|\sec_M|\le K$, then $|H|\le(n+1)\sqrt K$.
\end{corollary}
\begin{proof}
$|\sec_M|\le K$ implies $\sec_M\ge-K$; apply Theorem~A with $K_0=K$.
\end{proof}

\subsection{Strategy and organization}
Fix $x_0\in M$ and let $r:=\operatorname{dist}_M(x_0,\cdot)$. For $c_0>\sqrt{K_0}$, we consider the vector field $V=G^sN+\frac{2s}nHG^{s-1}\nabla_NN$, $G=(H_+^2-c_0^2)_+$, whose divergence controls a positive quantity $Y$ supported in $\{H>c_0\}$ (Lemma~\ref{lem:divV}). Testing $\divg V\ge Y$ against powers of cutoff functions of $r$ and iterating the resulting flow inequality shows that if $\{H>c_0\}\ne\emptyset$, the integral of $Y$ over $\{r\le R\}$ must grow exponentially in $R$ at a rate that, for $c_0>c_*(n,K_0;C_1)$, exceeds the rate $n\sqrt{K_0}$ allowed by the Bishop–Gromov inequality. Here $C_1$ is an upper bound for $\Hess r$ far from $x_0$: the uncontrolled field $\nabla_NN$ is traded for second derivatives of the cutoff functions (Lemma~\ref{lem:T}), and only an upper bound for $\Hess r$ is needed. Since $r$ is not smooth, this bound is used in the support sense of Calabi; Proposition~\ref{prop:cutoff} shows that it suffices for the integration by parts. By the Hessian comparison theorem, one can take $C_1=\sqrt{K_0}\coth(\sqrt{K_0}R_*)$ outside $B_{R_*}(x_0)$, and letting $R_*\to\infty$ gives the constant $c_*(n,K_0;\sqrt{K_0})=(n+1)\sqrt{K_0}$.

Section~\ref{sec:prelim} collects the identities relating the foliation to the ambient divergence, the Hessian comparison for the distance function in the support sense, the regularization argument, and the flow field. Section~\ref{sec:proof} proves Theorem~A.

\section{Preliminaries}\label{sec:prelim}

\subsection{Notation and identities}
We write $\nabla$, $g(\cdot,\cdot):=\langle\cdot,\cdot\rangle$, $\divg$, $R$ for the Levi-Civita connection, metric, divergence, and curvature tensor of $M$, with $R(X,Y)Z=\nabla_X\nabla_YZ-\nabla_Y\nabla_XZ-\nabla_{[X,Y]}Z$, so that $\sec(X,Y)=\langle R(X,Y)Y,X\rangle$ for orthonormal $X,Y$. Extend the shape operator to $TM$ by $A:=-\nabla N$. Since $\langle\nabla_XN,N\rangle=\frac12X\langle N,N\rangle=0$, $A$ takes values in $N^\perp$. For $X,Y$ tangent to a leaf $L$, $\langle AX,Y\rangle=\langle N,\nabla_XY\rangle$, which is symmetric in $X,Y$ because $[X,Y]$ is tangent to $L$; hence $A|_{TL}$ is a symmetric endomorphism of $TL$ and, computing $\divg N$ in an orthonormal frame $\{e_1,\dots,e_n,N\}$,
\[
\tr A|_{TL}=-\sum_i\langle\nabla_{e_i}N,e_i\rangle=-\divg N+\langle\nabla_NN,N\rangle=nH .
\]
Put
\[
\mathfrak{g}:=N(H),\qquad T:=\nabla_NN,\qquad |A|^2:=\tr\big((A|_{TL})^2\big),
\]
and let $\gl$, $\nabla^L$, $\divg_L$, and $\lapL$ denote the gradient, covariant derivative, divergence, and Laplacian on a leaf, with respect to its induced metric, acting on the restrictions of the corresponding objects to the leaves. Then $T$ is tangent to the leaves, $AN=-T$, and if $k_1,\dots,k_n$ are the eigenvalues of $A|_{TL}$,
\begin{equation}\label{eq:CS}
|A|^2=\sum_ik_i^2\ \ge\ \frac1n\Big(\sum_ik_i\Big)^2=nH^2 .
\end{equation}

\begin{lemma}\label{lem:gradH}
If $\kappa\in C^1(M)$ is constant along the leaves, then $\nabla\kappa=N(\kappa)N$. In particular $\nabla H=\mathfrak{g}N$.
\end{lemma}
\begin{proof}
$X(\kappa)=0$ for every $X$ tangent to the leaves, so $\nabla\kappa$ is parallel to $N$.
\end{proof}
The following lemma was proved by Barbosa, Kenmotsu, and Oshikiri in \cite{BKO}
\begin{lemma}[Proposition 2.14 in \cite{BKO}]\label{lem:divT}
For a $C^1$ vector field $X$ tangent to the leaves, $\divg X=\divg_LX-\langle X,T\rangle$ and 
$\divg T=|A|^2+\Ric(N,N)-n\mathfrak{g}$.
\end{lemma}
The following lemma was proved by Sampaio and Silva (see \cite{SS2026}) using similar ideas. For the sake of completeness, we provide a proof.
\begin{lemma}[Lemma 2.5 in \cite{SS2026}]\label{lem:divX}
 For $f\in C^2(M)$,
\[
\lapL f=\tr_{TL}\Hess f+nH\,N(f),\qquad \tr_{TL}\Hess f:=\sum_{i=1}^n\Hess f(e_i,e_i),
\]
where $\{e_i\}$ is an orthonormal basis of the tangent space of the leaf.
\end{lemma}
\begin{proof}
In a frame $\{e_1,\dots,e_n,N\}$ as above, the $e_i$ pertains to local vector fields tangent to the leaves. Then $\lapL f=\sum_i\big(e_ie_if-(\nabla^L_{e_i}e_i)f\big)$, $\Hess f(e_i,e_i)=e_ie_if-(\nabla_{e_i}e_i)f$, and $\nabla_{e_i}e_i=\nabla^L_{e_i}e_i+\langle\nabla_{e_i}e_i,N\rangle N=\nabla^L_{e_i}e_i+\langle Ae_i,e_i\rangle N$. Summing over $i$ gives $\tr_{TL}\Hess f=\lapL f-nH\,N(f)$.
\end{proof}

The following lemma was proved in \cite[Lemma 3.2]{SS2026} when $M=\mathbb{R}^{n+1}$. The proof works along the same lines, but since it is short, for completeness, we present it here.
\begin{lemma}[trading $T$ for Hessians]\label{lem:T}
For $\chi\in C^2_c(M)$ and $\kappa\in C^1(M)$ constant along the leaves,
\begin{equation}\label{eq:T}
\int_M\kappa\langle\nabla\chi,T\rangle=\int_M\kappa\,\tr_{TL}\Hess\chi+n\int_MH\kappa\,N(\chi).
\end{equation}
\end{lemma}
\begin{proof}
The field $\gl\chi:=\nabla\chi-N(\chi)N$ is $C^1$, compactly supported and tangent to the leaves, and its restriction to each leaf is the gradient of $\chi|_L$, so $\divg_L(\gl\chi)=\lapL\chi$. Since $T$ is tangent to the leaves, $\langle\gl\chi,T\rangle=\langle\nabla\chi,T\rangle$. By Lemmas~\ref{lem:divT} and \ref{lem:divX},
\[
\langle\nabla\chi,T\rangle=\lapL\chi-\divg(\gl\chi)=\tr_{TL}\Hess\chi+nH\,N(\chi)-\divg(\gl\chi).
\]
Multiply by $\kappa$. By Lemma~\ref{lem:gradH}, $\langle\nabla\kappa,\gl\chi\rangle=0$, so $\kappa\divg(\gl\chi)=\divg(\kappa\gl\chi)$, and the divergence of the compactly supported $C^1$ field $\kappa\gl\chi$ has integral zero.
\end{proof}

\subsection{The distance function and the Hessian comparison in the support sense}
The following notion goes back to Calabi \cite{Cal}; see also \cite{Wu}.

\begin{definition}\label{def:support}
Let $W\subset M$ be open, $f\in C^0(W)$, $x\in W$ and $a\in\R$. We say that $\Hess f(x)\le a\langle\cdot,\cdot\rangle$ \emph{in the support sense} if for every $\varepsilon>0$ there are a neighbourhood $U\subset W$ of $x$ and $\varphi\in C^2(U)$ (an \emph{upper support function} of $f$ at $x$) such that
\[
\varphi\ge f\ \text{ on }U,\qquad \varphi(x)=f(x),\qquad \Hess\varphi(x)\le(a+\varepsilon)\langle\cdot,\cdot\rangle .
\]
\end{definition}

If $f$ is $C^2$ near $x$, this is equivalent to $\Hess f(x)\le a\langle\cdot,\cdot\rangle$: $f$ itself is an upper support function, and conversely $\varphi-f$ has a local minimum at $x$, so $\Hess f(x)\le\Hess\varphi(x)$.

For \(K_0\ge 0\) and \(t>0\), let $\ct_{K_0}(t)$ be the function defined by:
\[
\operatorname{ct}_{K_0}(t)=
\begin{cases}
\sqrt{K_0}\coth(\sqrt{K_0}\,t),
& K_0>0,\\[4pt]
\dfrac1t,&K_0=0.
\end{cases}
\]
The function $\ct_{K_0}$ is continuous and decreasing on $(0,\infty)$. Moreover,
\[
\lim_{t\to\infty}\ct_{K_0}(t)=
\begin{cases}
\sqrt{K_0}, & \text{if } K_0>0,\\
0, & \text{if } K_0=0.
\end{cases}
\]
We denote by $\Cut(p)$ the cut locus of $p\in M$ and define the distance function $r_p\colon M\to[0,\infty)$ by
\[
r_p(x):=\operatorname{dist}_M(p,x).
\]
\begin{theorem}[Hessian comparison; see \cite{Pet}, \cite{GWpole}, \cite{EH}]\label{thm:hesscomp}
Let $M$ be complete with $\sec_M\ge-K_0$, $p\in M$ and $y\in M\setminus(\Cut(p)\cup\{p\})$. Then $r_p$ is smooth near $y$ and
\[
\Hess r_p(y)\le\ct_{K_0}\big(r_p(y)\big)\big(\langle\cdot,\cdot\rangle-dr_p\otimes dr_p\big)\le\ct_{K_0}\big(r_p(y)\big)\langle\cdot,\cdot\rangle .
\]
\end{theorem}

We also use two standard facts about cut loci (see \cite[Ch.~XIII, \S2]{dC}, \cite[Ch.~5]{CE}):
\begin{itemize}
\item[(C1)] if $\gamma\colon[0,\ell]\to M$ is a minimizing geodesic from $p=\gamma(0)$, then $\gamma(t)\notin\Cut(p)$ for $0<t<\ell$;
\item[(C2)] $q\in\Cut(p)$ if and only if $p\in\Cut(q)$.
\end{itemize}

The following lemma follows from Calabi's trick in \cite{Cal}, see also \cite{Wu} and \cite{EH}. For the sake of completeness, we provide a proof.
\begin{lemma}\label{lem:calabi}
Let $M$ be complete with $\sec_M\ge-K_0$, $x_0\in M$, and $r:=\operatorname{dist}_M(x_0,\cdot)$. Then for every $x\in M\setminus\{x_0\}$,
\[
\Hess r(x)\le\ct_{K_0}\big(r(x)\big)\langle\cdot,\cdot\rangle
\]
in the support sense. In particular, for every $R_*>0$, $\Hess r\le\ct_{K_0}(R_*)\langle\cdot,\cdot\rangle$ in the support sense at every point of $\{r>R_*\}$.
\end{lemma}
\begin{proof}
Fix $x\in M\setminus\{x_0\}$ and let $r(x)=\operatorname{dist}_M(x_0,x)>0.$ By completeness and the Hopf-Rinow theorem, there exists a unit-speed minimizing geodesic
$\gamma\colon[0,r(x)]\to M$ such that $\gamma(0)=x_0$ and $\gamma(r(x))=x$. Fix $\varepsilon\in(0,r(x))$ and set
\[
p_\varepsilon:=\gamma(\varepsilon),
\qquad
f_\varepsilon:=\varepsilon+r_{p_\varepsilon},
\]
where $r_{p_\varepsilon}(y):=\operatorname{dist}_M(p_\varepsilon,y)$. By the triangle inequality, for every $y\in M$,
\[
r(y)=\operatorname{dist}_M(x_0,y)
\leq \operatorname{dist}_M(x_0,p_\varepsilon)
+\operatorname{dist}_M(p_\varepsilon,y)
=f_\varepsilon(y),
\]
since $\operatorname{dist}_M(x_0,p_\varepsilon)=\varepsilon$.
Moreover, as $\gamma$ is minimizing, its restriction to
$[\varepsilon,r(x)]$ is also minimizing. Hence
\[
r_{p_\varepsilon}(x)=r(x)-\varepsilon,
\]
and consequently
\[
f_\varepsilon(x)=\varepsilon+r_{p_\varepsilon}(x)=r(x).
\]
Thus $f_\varepsilon\geq r$ on $M$ and $f_\varepsilon(x)=r(x)$. We next show that $x\notin\Cut(p_\varepsilon)$. Consider the reversed geodesic
\[
\widetilde{\gamma}(t):=\gamma(r(x)-t),
\qquad t\in[0,r(x)].
\]
It is a minimizing geodesic from $x$ to $x_0$, and
$p_\varepsilon=\widetilde{\gamma}(r(x)-\varepsilon)$ is an interior
point of this geodesic. By property~(C1), $p_\varepsilon\notin\Cut(x)$. The symmetry of the cut locus, property~(C2), then yields $x\notin\Cut(p_\varepsilon).$ Since $x\neq p_\varepsilon$, the distance function
$r_{p_\varepsilon}$ is smooth in a neighborhood of $x$. Therefore,
$f_\varepsilon$ is smooth near $x$, and
\[
\operatorname{Hess}f_\varepsilon(x)
=\operatorname{Hess}r_{p_\varepsilon}(x).
\]
Applying the Hessian comparison theorem at $x$, with base point
$p_\varepsilon$, gives
\[
\operatorname{Hess}f_\varepsilon(x)
\leq
\ct_{K_0}\big(r_{p_\varepsilon}(x)\big)\langle\cdot,\cdot\rangle_x
=
\ct_{K_0}\big(r(x)-\varepsilon\big)\langle\cdot,\cdot\rangle_x.
\]

Now let $\delta>0$ be arbitrary. Since $r(x)>0$ and
$\ct_{K_0}$ is continuous on $(0,\infty)$, we may choose
$\varepsilon\in(0,r(x))$ sufficiently small so that
\[
\ct_{K_0}\big(r(x)-\varepsilon\big)
\leq \ct_{K_0}\big(r(x)\big)+\delta.
\]
For this choice of $\varepsilon$, the smooth function
$f_\varepsilon$ satisfies
\[
f_\varepsilon\geq r,\qquad f_\varepsilon(x)=r(x),
\]
and
\[
\operatorname{Hess}f_\varepsilon(x)
\leq
\big(\ct_{K_0}(r(x))+\delta\big)\langle\cdot,\cdot\rangle_x.
\]
Since $\delta>0$ is arbitrary, the definition of the support sense
implies that
\[
\operatorname{Hess}r(x)
\leq \ct_{K_0}(r(x))\langle\cdot,\cdot\rangle_x
\]
in the support sense. As $x\in M\setminus\{x_0\}$ was arbitrary,
the first assertion follows.

Finally, let $R_*>0$ and suppose that $r(x)>R_*$. Since
$\ct_{K_0}$ is decreasing,
\[
\ct_{K_0}(r(x))\leq \ct_{K_0}(R_*).
\]
The first assertion therefore implies
\[
\operatorname{Hess}r(x)
\leq \ct_{K_0}(R_*)\langle\cdot,\cdot\rangle_x
\]
in the support sense at every $x\in\{r>R_*\}$. This proves the
second assertion.
\end{proof}

\subsection{Regularization and the cutoff inequality}

The following result follows from the divergence theorem and its extension to vector fields with locally Lipschitz components; see
\cite{PigolaSetti,EvansGariepy2015,Federer1969}.
\begin{lemma}\label{lem:lipdiv}
Let $W$ be a compactly supported vector field on $M$ whose components in local coordinates are locally Lipschitz. Then $\divg W\in L^\infty(M)$ is defined almost everywhere and $\int_M\divg W=0$.
\end{lemma}

\begin{remark}\label{rem:GW}
Lemmas~\ref{lem:1d} and~\ref{lem:reg} below are essentially known.
For $\beta\equiv0$, Lemma~\ref{lem:1d} is the classical fact that a
continuous function of one variable admitting upper support functions
with arbitrarily small second derivative is concave; see
\cite[Section~1, Lemmas~2 and~4]{Wu}, where the proofs are omitted as
standard. The general case reduces to this one by subtracting a $C^2$
function $B$ with $B''=\beta$.

Lemma~\ref{lem:reg} is a variant, with Euclidean convolution in normal
coordinates in place of Riemannian convolution, of the smoothing
results of Greene and Wu \cite[\S2]{GW79}. Indeed, if $f$ is locally
$L$-Lipschitz and $\operatorname{Hess} f\le a\langle\cdot,\cdot\rangle$
in the support sense on $W$, then for every $\eta>0$ the function $-f$
is $(-a-\eta)$-convex on $W$ in the sense of \cite[p.~60]{GW79}. To see
this, apply \cite[Lemmas~2 and~4]{Wu} near each $p\in W$ to
$-f+\tfrac12\,(a+\eta/2)\operatorname{dist}_M(p,\cdot)^2$.
Hence, by \cite[pp.~57, 59, 60--61]{GW79}, the Riemannian convolutions
$f_\varepsilon$ of $f$ are smooth near any compact set $K\subset W$,
converge to $f$ uniformly on $K$, and, for all sufficiently small
$\varepsilon$, satisfy $|\nabla f_\varepsilon|<L+\eta$ and
$\operatorname{Hess} f_\varepsilon\le (a+\eta)\langle\cdot,\cdot\rangle$
near $K$ (see also \cite{GW74}).

We also use the almost everywhere convergence
$\nabla f_\varepsilon\to\nabla f$, which is not stated in \cite{GW79}.
It follows from the uniform semiconcavity of the approximations in
local charts and \cite[Theorem~24.5]{Rockafellar}. Since we need these
facts in the precise form stated above, we give self-contained proofs
of Lemmas~\ref{lem:1d} and~\ref{lem:reg} in Appendix~\ref{app:reg}.
\end{remark}

\begin{lemma}\label{lem:1d}
Let $I\subset\mathbb{R}$ be an open interval, and let
$v,\beta\in C^0(I)$. Assume that for every $t_0\in I$ and every
$\varepsilon>0$, there exists a $C^2$ function $\phi$ defined on a
neighborhood of $t_0$ such that
\[
\phi\geq v,\qquad \phi(t_0)=v(t_0),\qquad
\phi''(t_0)\leq\beta(t_0)+\varepsilon.
\]
Then $v''\leq\beta$ in the sense of distributions on $I$; that is,
\[
\int_I v(t)\zeta''(t)\,dt
\leq
\int_I \beta(t)\zeta(t)\,dt
\]
for every nonnegative $\zeta\in C_c^\infty(I)$.
\end{lemma}

Fix a mollifier $\mu\in C^\infty_c(\R^{n+1})$, $\mu\ge0$, $\supp\mu\subset B(0,1)$, $\int\mu=1$, and put $\mu_\delta(z):=\delta^{-n-1}\mu(z/\delta)$. For a symmetric matrix $S$ we write $\|S\|$ for its operator norm, and for the Christoffel symbols in a chart we put $\|\Gamma(y)\|:=\sup\{|\Gamma^k_{ij}(y)\xi^i\xi^j\zeta_k|:|\xi|=|\zeta|=1\}$ (Euclidean norms).

\begin{lemma}[regularization in normal coordinates]\label{lem:reg}
Let $W\subset M$ be open, $a\in\R$, $L>0$, and let $f\in C^0(W)$ be such that every point of $W$ has a neighbourhood on which $f$ is $L$-Lipschitz, and $\Hess f\le a\langle\cdot,\cdot\rangle$ in the support sense at every point of $W$. Let $K\subset W$ be compact and $\eta>0$. Then there are $\sigma>0$ and points $c_1,\dots,c_m\in K$ such that, with $U_i:=B_{2\sigma}(c_i)$:
\begin{enumerate}
\item[\textup{(i)}] $K\subset U_1\cup\dots\cup U_m$, and $\exp_{c_i}$ maps the Euclidean ball $B(0,4\sigma)\subset T_{c_i}M\cong\R^{n+1}$ (identified by an orthonormal basis) diffeomorphically onto the geodesic ball $B_{4\sigma}(c_i)\subset W$;
\item[\textup{(ii)}] for $0<\delta<\sigma$ the functions $f_{i,\delta}:=(u_i*\mu_\delta)\circ\exp_{c_i}^{-1}$, $u_i:=f\circ\exp_{c_i}$, are smooth on $B_{3\sigma}(c_i)$, $|\nabla f_{i,\delta}|\le2L$ there, $f_{i,\delta}\to f$ uniformly on $\overline U_i$ and $\nabla f_{i,\delta}\to\nabla f$ a.e. on $U_i$ as $\delta\to0$;
\item[\textup{(iii)}] there are $\omega_i(\delta)\to0$ such that $\Hess f_{i,\delta}\le(a+\eta+\omega_i(\delta))\langle\cdot,\cdot\rangle$ on $\overline U_i$.
\end{enumerate}
\end{lemma}

The next proposition is the form of Lemma~\ref{lem:T} that will be used with the cutoff functions of Section~\ref{sec:proof}.

\begin{proposition}[cutoff inequality]\label{prop:cutoff}
Let $M$ be complete with $\sec_M\ge-K_0$, $x_0\in M$, $r:=\operatorname{dist}_M(x_0,\cdot)$, $0<R_*<R_1<R_2$ and $C_1:=\ct_{K_0}(R_*)$. Let $\Phi\in C^\infty(\R)$ satisfy $\Phi'\le0$, $\Phi\equiv\Phi_-$ on $(-\infty,R_1]$ for a constant $\Phi_-$, and $\Phi\equiv0$ on $[R_2,\infty)$, and let $\lambda\in C^1(M)$, $\lambda\ge0$, be constant along the leaves. Then $\varphi:=\Phi\circ r$ is Lipschitz with compact support, $\nabla\varphi=\Phi'(r)\nabla r$ a.e., and
\begin{equation}\label{eq:cutoff}
\int_M\lambda\langle\nabla\varphi,T\rangle-n\int_MH\lambda\,N(\varphi)\ \ge\ \int_M\lambda\Big[\Phi''(r)\,|\gl r|^2+nC_1\,\Phi'(r)\Big],
\end{equation}
where $\gl r:=\nabla r-N(r)N$ is defined a.e., and the integrand on the right vanishes outside $S:=\{R_1\le r\le R_2\}$.
\end{proposition}

\begin{proof}
By the Hopf-Rinow theorem, the closed metric balls of $M$ are
compact. Since the distance function
$r=\operatorname{dist}_M(x_0,\cdot)$ is $1$-Lipschitz, the function
$\varphi=\Phi\circ r$ is Lipschitz. Moreover, $\varphi=0$ on
$\{r\ge R_2\}$, and hence $\varphi$ has compact support. The chain
rule for Lipschitz functions gives
\[
\nabla\varphi=\Phi'(r)\nabla r
\quad\text{almost everywhere on }M.
\]
In particular, $\nabla\varphi=0$ almost everywhere outside $S:=\{x\in M:R_1\le r(x)\le R_2\}.$ The set $S$ is compact and is contained in the open set
$W:=\{r>R_*\}$. For every compactly supported Lipschitz function $\psi$, set
\[
\mathcal Q(\psi):=
\int_M\lambda\langle\nabla\psi,T\rangle
-n\int_M H\lambda\,N(\psi).
\]
We first establish the desired inequality by approximation.

By Lemma~\ref{lem:calabi}, the distance function satisfies
\[
\operatorname{Hess}r\le C_1g
\]
in the support sense on $W$, where
$C_1=\ct_{K_0}(R_*)$. Fix $\eta>0$. Applying
Lemma~\ref{lem:reg} to $f=r$, $L=1$, $a=C_1$, and the compact set
$S$, we obtain finitely many open sets
$U_i=B_{2\sigma}(c_i)$ covering $S$, with
$\overline{U_i}\subset W$, and smooth functions
$r_{i,\delta}$ on $B_{3\sigma}(c_i)$, for all sufficiently small
$\delta>0$, such that
\[
r_{i,\delta}\longrightarrow r
\quad\text{uniformly on }\overline{U_i},
\]
\[
\nabla r_{i,\delta}\longrightarrow\nabla r
\quad\text{almost everywhere on }U_i,
\]
and
\[
\operatorname{Hess}r_{i,\delta}
\le \bigl(C_1+\eta+\omega_i(\delta)\bigr)g
\quad\text{on }\overline{U_i},
\qquad
\omega_i(\delta)\longrightarrow0.
\]
Here and below, $\delta$ is taken sufficiently small that all the
regularizations are defined on the indicated sets.

Let $O_-:=\{r<R_1\}$, $O_+:=\{r>R_2\}$. Choose a smooth partition of unity
\[
\zeta_-+\zeta_++\sum_{i=1}^m\zeta_i=1
\]
subordinate to the cover
$\{O_-,O_+,U_1,\ldots,U_m\}$, with
$\operatorname{supp}\zeta_-$ and
$\operatorname{supp}\zeta_i$ compact. Define
\[
\varphi_\delta
:=\zeta_-\Phi_-+
\sum_{i=1}^m\zeta_i\,\Phi(r_{i,\delta}).
\]
Each product $\zeta_i\Phi(r_{i,\delta})$ is understood on $U_i$ and
extended by zero outside $U_i$; this extension is smooth because
$\operatorname{supp}\zeta_i\Subset U_i$. Thus
$\varphi_\delta\in C_c^\infty(M)$, and all the functions
$\varphi_\delta$ have support in a fixed compact subset of $M$.

Since $\Phi(r)=\Phi_-$ on $O_-$ and $\Phi(r)=0$ on $O_+$, the
partition of unity gives
\[
\varphi=\zeta_-\Phi_-+
\sum_{i=1}^m\zeta_i\,\Phi(r).
\]
The uniform convergence $r_{i,\delta}\to r$, together with the
boundedness of $\Phi'$, implies
$\varphi_\delta\to\varphi$ uniformly. Moreover, the chain rule gives
\[
\nabla\bigl(\Phi(r_{i,\delta})\bigr)
=\Phi'(r_{i,\delta})\nabla r_{i,\delta}.
\]
Consequently, the almost-everywhere convergence of
$\nabla r_{i,\delta}$ and the uniform convergence of
$r_{i,\delta}$ imply
\[
\nabla\varphi_\delta\longrightarrow\nabla\varphi
\quad\text{almost everywhere on }M.
\]
The gradients $\nabla\varphi_\delta$ are uniformly bounded, and
their supports lie in a fixed compact set. Since $\lambda$, $T$,
and $H$ are bounded on that compact set, the dominated convergence
theorem yields
\begin{equation}\label{eq:Q-convergence-cutoff}
\mathcal Q(\varphi_\delta)\longrightarrow\mathcal Q(\varphi).
\end{equation}

By Lemma~\ref{lem:T}, applied to the smooth compactly supported
function $\varphi_\delta$, we have
\[
\mathcal Q(\varphi_\delta)
=\int_M\lambda\,\operatorname{tr}_{TL}
\operatorname{Hess}\varphi_\delta.
\]
The product rule for the Hessian gives
\[
\operatorname{Hess}\varphi_\delta
=\sum_{i=1}^m\zeta_i
\operatorname{Hess}\bigl(\Phi(r_{i,\delta})\bigr)
+E_\delta,
\]
where
\[
\begin{aligned}
E_\delta
={}&\Phi_-\operatorname{Hess}\zeta_-
+\sum_{i=1}^m\Bigl[
d\zeta_i\otimes d\bigl(\Phi(r_{i,\delta})\bigr)
+d\bigl(\Phi(r_{i,\delta})\bigr)\otimes d\zeta_i
+\Phi(r_{i,\delta})\operatorname{Hess}\zeta_i
\Bigr].
\end{aligned}
\]
By the convergence already established, $E_\delta$ converges almost everywhere to
\[
\begin{aligned}
E_0
={}&\Phi_-\operatorname{Hess}\zeta_-
+\sum_{i=1}^m\Bigl[
d\zeta_i\otimes d\varphi
+d\varphi\otimes d\zeta_i
+\varphi\operatorname{Hess}\zeta_i
\Bigr].
\end{aligned}
\]
Furthermore, $E_\delta$ is uniformly bounded and supported in a
fixed compact set.

We claim that $E_0=0$ almost everywhere. Indeed,
\[
\sum_{i=1}^m\zeta_i=1-\zeta_--\zeta_+,
\]
so
\[
\sum_{i=1}^m d\zeta_i=-d\zeta_--d\zeta_+,
\qquad
\sum_{i=1}^m\operatorname{Hess}\zeta_i
=-\operatorname{Hess}\zeta_-
-\operatorname{Hess}\zeta_+.
\]
Since $\operatorname{supp}\zeta_-\subset O_-$ and
$\operatorname{supp}\zeta_+\subset O_+$, we have
\[
d\varphi=0\quad\text{almost everywhere on }
\operatorname{supp}d\zeta_-\cup\operatorname{supp}d\zeta_+,
\]
and $\varphi=\Phi_-$ on $\operatorname{supp}\zeta_-$, $\varphi=0$ on $\operatorname{supp}\zeta_+$. It follows that the terms involving $d\varphi$ vanish almost
everywhere, while
\[
\begin{aligned}
E_0
=\Phi_-\operatorname{Hess}\zeta_-
-\varphi\operatorname{Hess}\zeta_-
-\varphi\operatorname{Hess}\zeta_+
=0
\quad\text{almost everywhere}.
\end{aligned}
\]
The dominated convergence theorem therefore gives
\begin{equation}\label{eq:error-convergence-cutoff}
\int_M\lambda\,\operatorname{tr}_{TL}E_\delta
\longrightarrow0.
\end{equation}

It remains to estimate the main Hessian terms. For each $i$, the
chain rule for the Hessian yields
\[
\begin{aligned}
\operatorname{Hess}\bigl(\Phi(r_{i,\delta})\bigr)
={}&\Phi''(r_{i,\delta})\,
dr_{i,\delta}\otimes dr_{i,\delta}+\Phi'(r_{i,\delta})\operatorname{Hess}r_{i,\delta}.
\end{aligned}
\]
Taking the trace over the tangent space $TL$ of a leaf, whose
dimension is $n$, gives
\[
\begin{aligned}
\operatorname{tr}_{TL}
\operatorname{Hess}\bigl(\Phi(r_{i,\delta})\bigr)
={}&\Phi''(r_{i,\delta})
\bigl|\nabla_Lr_{i,\delta}\bigr|^2+\Phi'(r_{i,\delta})
\operatorname{tr}_{TL}\operatorname{Hess}r_{i,\delta},
\end{aligned}
\]
where
$
\nabla_Lr_{i,\delta}
:=\nabla r_{i,\delta}
-N(r_{i,\delta})N.
$
By Lemma~\ref{lem:reg}(iii),
\[
\operatorname{tr}_{TL}\operatorname{Hess}r_{i,\delta}
\le n\bigl(C_1+\eta+\omega_i(\delta)\bigr).
\]
Because $\Phi'\le0$, multiplying this inequality by
$\Phi'(r_{i,\delta})$ reverses its direction. Hence
\[
\begin{aligned}
\operatorname{tr}_{TL}
\operatorname{Hess}\bigl(\Phi(r_{i,\delta})\bigr)
\ge{}&
\Phi''(r_{i,\delta})
\bigl|\nabla_Lr_{i,\delta}\bigr|^2+n\Phi'(r_{i,\delta})
\bigl(C_1+\eta+\omega_i(\delta)\bigr).
\end{aligned}
\]
Multiplying by $\lambda\zeta_i\ge0$, integrating, and summing
over $i$, we obtain
\[
\begin{aligned}
\mathcal Q(\varphi_\delta)
\ge{}&
\sum_{i=1}^m\int_M
\lambda\zeta_i\,
\Phi''(r_{i,\delta})
\bigl|\nabla_Lr_{i,\delta}\bigr|^2\\
&+n\sum_{i=1}^m\int_M
\lambda\zeta_i\,\Phi'(r_{i,\delta})
\bigl(C_1+\eta+\omega_i(\delta)\bigr)\\
&+\int_M\lambda\,\operatorname{tr}_{TL}E_\delta.
\end{aligned}
\]

For almost every point of $U_i$, the convergence of the gradients
implies
\[
\nabla_Lr_{i,\delta}\longrightarrow
\nabla_Lr:=\nabla r-N(r)N.
\]
All terms in the preceding integrals are dominated by integrable
functions independent of sufficiently small $\delta$: this follows
from the uniform gradient bounds in Lemma~\ref{lem:reg}(ii), the
boundedness of $\Phi'$ and $\Phi''$, the convergence
$\omega_i(\delta)\to0$, and the compactness of the supports.
Therefore, by dominated convergence, \eqref{eq:Q-convergence-cutoff},
and \eqref{eq:error-convergence-cutoff}, we obtain
\[
\begin{aligned}
\mathcal Q(\varphi)\ge
\sum_{i=1}^m\int_M\lambda\zeta_i
\Bigl[
\Phi''(r)|\nabla_Lr|^2+n\Phi'(r)(C_1+\eta)
\Bigr].
\end{aligned}
\]
Since $\Phi'$ and $\Phi''$ vanish outside $S$, and
$\zeta_-=\zeta_+=0$ on $S$, we have
$\sum_{i=1}^m\zeta_i=1$ on $S$. Recalling that
$\nabla^Lr=\nabla r-N(r)N=\gl r$ almost everywhere, it follows
that
\[
\mathcal Q(\varphi)\ge
\int_M\lambda\Bigl[
\Phi''(r)|\gl r|^2
+n\Phi'(r)(C_1+\eta)
\Bigr].
\]
Finally, letting $\eta\to0^+$ and using dominated convergence
on the compact set $S$, we conclude that
\[
\mathcal Q(\varphi)\ge
\int_M\lambda\Bigl[
\Phi''(r)|\gl r|^2+nC_1\Phi'(r)
\Bigr],
\]
which is precisely \eqref{eq:cutoff}.
\end{proof}

\begin{remark}
If $r$ were $C^2$ on a neighbourhood of $S$, Proposition~\ref{prop:cutoff} would follow directly from Lemma~\ref{lem:T}, since then $\varphi\in C^2_c(M)$ and $\tr_{TL}\Hess\varphi=\Phi''(r)|\gl r|^2+\Phi'(r)\tr_{TL}\Hess r\ge\Phi''(r)|\gl r|^2+nC_1\Phi'(r)$. The approximation argument is needed because $r$ is not differentiable on the cut locus of $x_0$. Alternatively one could smooth $r$ by the Riemannian convolution of Greene and Wu \cite{GW79}.
\end{remark}

\subsection{The flow field}\label{subsec:flow}

Fix a constant $c_0>\sqrt{K_0}$ and an integer $s\ge3$ such
that $2s>n$. Set $\delta_0^2:=c_0^2-K_0>0.$ For $h\in\mathbb R$, define
\[
\phi(h):=
\begin{cases}
0, & h\le c_0,\\
h^2-c_0^2, & h>c_0.
\end{cases}
\]
Since $c_0>0$, this is equivalent to
$\phi(h)=((h_+)^2-c_0^2)_+$, where
$h_+:=\max\{h,0\}$. In particular, $\phi$ is nonnegative and locally Lipschitz, and it is smooth on $\mathbb{R}\setminus\{c_0\}$.

Define
\begin{equation*}\label{eq:G}
G:=\phi(H)=\bigl((H_+)^2-c_0^2\bigr)_+,
\qquad
F:=HG^s,
\qquad
Y:=(2s-n)HG^s+2s\delta_0^2HG^{s-1}.
\end{equation*}
By the definition of $\phi$, we have
\[
G(x)>0
\quad\Longleftrightarrow\quad
H(x)>c_0.
\]
Thus $G>0$ precisely on the open set
$
U:=\{x\in M:H(x)>c_0\},
$
and $G=0$ on $M\setminus U$. We next record the regularity of these functions. Since
\[
\phi(h)^s=(h-c_0)_+^s(h+c_0)^s,
\]
Leibniz's rule shows that, for every \(k=0,\ldots,s-1\), the \(k\)-th derivative on \((c_0,\infty)\) tends to zero as \(h\to c_0^+\), because each term contains a factor \((h-c_0)^{s-j}\) with \(j\le k<s\). Since \(\phi^s\equiv0\) on \((-\infty,c_0]\), these derivatives match continuously at \(h=c_0\). Hence
$
\phi^s\in C^{s-1}(\mathbb R).
$
Similarly, $\phi^{s-1}$ vanishes to order $s-1$ at $c_0$,
and hence
$
\phi^{s-1}\in C^{s-2}(\mathbb R).
$
Since $s\ge3$, both $\phi^s$ and $\phi^{s-1}$ are of class
$C^1$. Assuming, as in the present setting, that $H\in C^1(M)$,
the chain rule therefore implies that
\[
G^s=\phi(H)^s\in C^1(M),
\qquad
G^{s-1}=\phi(H)^{s-1}\in C^1(M).
\]
It follows that $F,Y\in C^1(M)$.

Furthermore, $F$ and $Y$ vanish identically on $M\setminus U$.
On $U$, we have $H>c_0>0$, $G>0$, and $\delta_0^2>0$.
Since $2s>n$, both coefficients $2s-n$ and $2s$ are positive.
Consequently,
\[
F\ge0,\qquad Y\ge0
\quad\text{on }M.
\]

Finally, define
\begin{equation}\label{eq:V}
V:=\kappa N+\lambda T,
\qquad
\kappa:=G^s,
\qquad
\lambda:=\frac{2s}{n}HG^{s-1}.
\end{equation}
The regularity established above implies that
$\kappa,\lambda\in C^1(M)$. Since $H$ is constant along each
leaf, and $\kappa$ and $\lambda$ are functions of $H$, both
$\kappa$ and $\lambda$ are constant along the leaves.
Moreover, $\lambda=0$ on $M\setminus U$, whereas on $U$ we
have $H>c_0>0$ and $G>0$. Therefore,
$
\lambda\ge0\quad\text{on }M.
$
In particular, provided $N$ and $T$ are of class $C^1$, the
vector field $V$ is of class $C^1$.

\begin{lemma}\label{lem:divV}
If $\Ric(N,N)\ge-nK_0$, then $\divg V\ge Y\ge0$ on $M$.
\end{lemma}
\begin{proof}
Recall that
\[
\kappa=G^s,
\qquad
\lambda=\frac{2s}{n}HG^{s-1},
\qquad
G=\bigl((H_+)^2-c_0^2\bigr)_+,
\]
where $c_0>0$, $s\ge3$, and $2s>n$. By the regularity
established above, $\kappa,\lambda\in C^1(M)$. Moreover,
differentiating $\kappa$ with respect to $H$ gives
\[
\kappa_H:=\frac{d\kappa}{dH}
=sG^{s-1}\frac{dG}{dH}
=2sHG^{s-1}
=n\lambda
\]
on $\{H>c_0\}$. On the complement of this set, both sides vanish. By continuity, we therefore have
$\kappa_H=n\lambda$ on $M$.

By Lemma~\ref{lem:gradH}, the gradients of $\kappa$ and
$\lambda$ satisfy
\[
\nabla\kappa=\kappa_H\,\mathfrak{g}N,
\qquad
\nabla\lambda=\lambda_H\,\mathfrak{g}N,
\]
where $\mathfrak{g}=N(H)$ denotes the scalar appearing in
Lemma~\ref{lem:gradH}. In particular,
$
\langle\nabla\lambda,T\rangle=0,
$
because $T$ is tangent to the leaves and hence orthogonal to
$N$. Using the product rule for divergence, we obtain
\[
\begin{aligned}
\operatorname{div}V
&=\operatorname{div}(\kappa N)
  +\operatorname{div}(\lambda T)\\
&=\langle\nabla\kappa,N\rangle
  +\kappa\operatorname{div}N
  +\langle\nabla\lambda,T\rangle
  +\lambda\operatorname{div}T.
\end{aligned}
\]
By the identities in \eqref{eq:H} and Lemma~\ref{lem:divT},
\[
\operatorname{div}N=-nH,
\qquad
\operatorname{div}T
=|A|^2+\operatorname{Ric}(N,N)-n\mathfrak{g}.
\]
Consequently,
\[
\begin{aligned}
\operatorname{div}V
&=\kappa_H \mathfrak{g}-nH\kappa
  +\lambda\bigl(|A|^2+\operatorname{Ric}(N,N)-n\mathfrak{g}\bigr)\\
&=(\kappa_H-n\lambda)\mathfrak{g}
  +\lambda\bigl(|A|^2+\operatorname{Ric}(N,N)\bigr)
  -nH\kappa.
\end{aligned}
\]
Since $\kappa_H=n\lambda$ on $M$, the terms involving $\mathfrak{g}$
cancel, and we conclude that
\begin{equation}\label{eq:divV-identity}
\operatorname{div}V
=\lambda\bigl(|A|^2+\operatorname{Ric}(N,N)\bigr)
-nH\kappa.
\end{equation}

We now prove the desired inequality separately on $M\setminus U$
and on $U$, where $U=\{H>c_0\}$.

On $M\setminus U$, we have $G=0$, and therefore
$\kappa=\lambda=0$. Since $\kappa$ is $C^1$ and vanishes on
this set, its derivative also vanishes at every interior point
of $M\setminus U$. At points of $\partial U$, the identity
$\kappa_H=n\lambda$ and continuity give the same conclusion.
It follows from \eqref{eq:divV-identity} that
$
\operatorname{div}V=0
\quad\text{on }M\setminus U.
$
By the definition of $Y$, we likewise have
$
Y=0
\quad\text{on }M\setminus U.
$
Hence $\operatorname{div}V=Y=0$ there. On $U$, we have $H>c_0>0$, $G>0$, and
$
\lambda=\frac{2s}{n}HG^{s-1}>0.
$
The Cauchy-Schwarz inequality for the principal curvatures
implies
$
|A|^2\ge nH^2
$. Combining this estimate with the assumed lower
bound $\operatorname{Ric}(N,N)\ge-nK_0$, we obtain
\[
|A|^2+\operatorname{Ric}(N,N)
\ge n(H^2-K_0).
\]
Since $\lambda\ge0$, identity \eqref{eq:divV-identity} yields
\[
\begin{aligned}
\operatorname{div}V
&\ge n\lambda(H^2-K_0)-nH\kappa\\
&=2sHG^{s-1}(H^2-K_0)-nHG^s.
\end{aligned}
\]
On $U$, the definition of $G$ gives
$
G=H^2-c_0^2.
$
Recalling that $\delta_0^2=c_0^2-K_0$, we have
\[
H^2-K_0=G+\delta_0^2.
\]
Substituting this identity into the preceding estimate, we find
\[
\begin{aligned}
\operatorname{div}V
&\ge 2sHG^{s-1}(G+\delta_0^2)-nHG^s\\
&=(2s-n)HG^s+2s\delta_0^2HG^{s-1}\\
&=Y.
\end{aligned}
\]
Finally, since $2s>n$, $\delta_0^2>0$, and $H,G>0$ on $U$,
the definition of $Y$ implies $Y\ge0$ on $U$. Together with
$Y=0$ on $M\setminus U$, this proves that
$\operatorname{div}V\ge Y\ge0$
on $M$.
\end{proof}

\section{Proof of Theorem A}\label{sec:proof}

Throughout this section, $M^{n+1}$ is complete with $\sec_M\ge-K_0$, $\mathfrak{F}$ is a codimension one CMC foliation of $M$, $x_0\in M$ is fixed and $r:=\operatorname{dist}_M(x_0,\cdot)$. If $v\in T_xM$ is a unit vector and $\{v,e_1,\dots,e_n\}$ is an orthonormal basis, then $\Ric(v,v)=\sum_i\sec(v,e_i)\ge-nK_0$; thus
\begin{equation}\label{eq:ric}
\Ric_M\ge-nK_0 ,
\end{equation}
which is what Lemma~\ref{lem:divV} and the volume comparison below require. We fix $R_*>0$, put $C_1:=\ct_{K_0}(R_*)$ and $R_\flat:=R_*+1$, and let $c_0,s,G,\delta_0,F,Y,\kappa,\lambda,V$ be as in Subsection \ref{subsec:flow}.

\subsection{Cutoff functions and volume}
Let $\psi\in C^\infty(\mathbb R)$ be a nonincreasing function such
that
$\psi(t)=1$ for $t\le0$, and $\psi(t)=0$ for $t\ge1$.
Set
\[
C_\psi:=\sup_{t\in\mathbb R}|\psi'(t)|,
\qquad
C'_\psi:=\sup_{t\in\mathbb R}|\psi''(t)|.
\]
Since $\psi(0)=1$ and $\psi(1)=0$, the mean value theorem implies that there exists $\xi\in(0,1)$ such that
$
|\psi'(\xi)|=1.
$
Consequently, $C_\psi\ge1$. Moreover, for every $\tau>0$,
one can choose such a smooth nonincreasing cutoff function $\psi$
with $C_\psi\le1+\tau$.

For $\varepsilon\in(0,1]$ and $R\ge R_\flat$, define
\[
\chi(x):=\psi\left(\frac{r(x)-R}{\varepsilon R}\right),
\qquad
A_R:=\{x\in M:R\le r(x)\le(1+\varepsilon)R\}.
\]
Since $r$ is $1$-Lipschitz and $\psi$ is smooth with bounded
derivative, $\chi$ is Lipschitz. The monotonicity and prescribed
values of $\psi$ imply
\[
0\le\chi\le1,\qquad
\chi=1\ \text{on }\{r\le R\},\qquad
\chi=0\ \text{on }\{r\ge(1+\varepsilon)R\}.
\]
In particular,
$
\operatorname{supp}\chi
\subset\{r\le(1+\varepsilon)R\}.
$
The chain rule for Lipschitz functions yields, almost everywhere,
\[
\nabla\chi
=\frac{1}{\varepsilon R}
\psi'\left(\frac{r-R}{\varepsilon R}\right)\nabla r.
\]
Since $|\nabla r|\le1$ almost everywhere, we obtain
\[
|\nabla\chi|\le\frac{C_\psi}{\varepsilon R}
\quad\text{almost everywhere on }M.
\]
Furthermore, $\psi'=0$ on $(-\infty,0)\cup(1,\infty)$,
so $\nabla\chi=0$ almost everywhere outside $A_R$.

Define
\begin{equation}\label{eq:CR}
C'_R:=nC_1+\frac{C'_\psi}{C_\psi\,\varepsilon R}.
\end{equation}
Since $C_\psi\ge1$ and $C'_\psi\ge0$ are independent of $R$,
the function $R\mapsto C'_R$ is nonincreasing and
$
\lim_{R\to\infty}C'_R=nC_1.
$

We next establish a volume bound. By the Ricci curvature
lower bound in \eqref{eq:ric}, the Bishop-Gromov volume
comparison theorem (see \cite{Pet}) gives
\[
\operatorname{vol}B_t(x_0)
\le |\mathbb S^n|
\int_0^t
\left(
\frac{\sinh(\sqrt{K_0}\,s)}{\sqrt{K_0}}
\right)^n ds,
\]
where the integrand is interpreted as $s^n$ when $K_0=0$.
Using $\sinh x\le xe^x$ for $x\ge0$, we find
\[
\left(
\frac{\sinh(\sqrt{K_0}\,s)}{\sqrt{K_0}}
\right)^n
\le s^n e^{n\sqrt{K_0}\,s}.
\]
Hence
\[
\begin{aligned}
\operatorname{vol}B_t(x_0)
&\le |\mathbb S^n|\int_0^t
s^n e^{n\sqrt{K_0}\,s}\,ds\\
&\le \frac{|\mathbb S^n|}{n+1}
t^{n+1}e^{n\sqrt{K_0}\,t}.
\end{aligned}
\]
Writing $\omega_n:=\frac{|\mathbb S^n|}{n+1}$, we conclude that
\[
\operatorname{vol}B_t(x_0)
\le \omega_n t^{n+1}e^{n\sqrt{K_0}\,t}.
\]

Since $\{r\le t\}\subset B_{t'}(x_0)$ for every $t'>t$,
letting $t'\downarrow t$ in the preceding estimate gives
\begin{equation}\label{eq:BG}
\operatorname{vol}\{r\le(1+\varepsilon)R\}
\le\mathcal V(R),
\end{equation}
where
$\mathcal V(R):=
\omega_n\bigl((1+\varepsilon)R\bigr)^{n+1}
e^{n\sqrt{K_0}(1+\varepsilon)R}.
$

Finally, define
\[
I_j:=\int_M\chi^jY\quad(j\ge1),\qquad
I_0:=\int_{\{r\le(1+\varepsilon)R\}}Y,
\qquad
J(t):=\int_{\{r\le t\}}Y.
\]
Because $Y$ is continuous and the closed metric balls are
compact by the Hopf-Rinow theorem, these integrals are finite.
Moreover, $Y\ge0$ by Lemma~\ref{lem:divV}, and $0\le\chi\le1$.
Thus, for every integer $j\ge1$,
\[
0\le\chi^j\le\chi^{j-1}\le1,
\]
where for $j=1$ we interpret $\chi^0=1$ on
$\{r\le(1+\varepsilon)R\}$ in the integral defining $I_0$.
Since $\chi=1$ on $\{r\le R\}$ and
$\operatorname{supp}\chi\subset\{r\le(1+\varepsilon)R\}$,
we obtain
\[
\begin{aligned}
J(R)
&=\int_{\{r\le R\}}Y\\
&\le \int_M\chi^jY\\
&=I_j\\
&\le I_{j-1}\\
&\le
\int_{\{r\le(1+\varepsilon)R\}}Y\\
&=I_0\\
&=J\bigl((1+\varepsilon)R\bigr).
\end{aligned}
\]
Therefore,
\begin{equation}\label{eq:Imono}
J(R)\le I_j\le I_{j-1}\le I_0
=J\bigl((1+\varepsilon)R\bigr),
\qquad j\ge1.
\end{equation}

\subsection{The flow inequality and its iteration}

\begin{proposition}[flow inequality]\label{prop:flux}
For $R\ge R_\flat$ and every integer $j\ge1$,
\begin{equation}\label{eq:flux}
I_j\ \le\ \frac{jC_\psi}{\varepsilon R}\int_{A_R}\chi^{j-1}\,\cB_R,\qquad \cB_R:=G^{s}+2sH^2G^{s-1}+\frac{2s}{n}C'_R\,HG^{s-1}.
\end{equation}
\end{proposition}

\begin{proof}
Fix $R\ge R_\flat$ and an integer $j\ge1$. Recall that
\[
V=\kappa N+\lambda T,
\qquad
\kappa=G^s,
\qquad
\lambda=\frac{2s}{n}HG^{s-1},
\]
and that $\operatorname{div}V\ge Y\ge0$ by
Lemma~\ref{lem:divV}. Since $\chi$ is Lipschitz and compactly
supported, the vector field $\chi^jV$ is also Lipschitz and
compactly supported. Lemma~\ref{lem:lipdiv} therefore yields
$
\int_M\operatorname{div}(\chi^jV)=0.
$
Using the product rule for the divergence, we obtain
\[
\begin{aligned}
\int_M\chi^j\operatorname{div}V
&=-\int_M\langle\nabla(\chi^j),V\rangle\\
&=-\int_M\kappa N(\chi^j)
  -\int_M\lambda\langle\nabla(\chi^j),T\rangle.
\end{aligned}
\]

We now apply Proposition~\ref{prop:cutoff} to the function
$
\Phi(t):=\psi\left(\frac{t-R}{\varepsilon R}\right)^j.
$
Since $\psi$ is nonincreasing, equals $1$ on $(-\infty,0]$,
and vanishes on $[1,\infty)$, the function $\Phi$ is smooth
and nonincreasing, with
\[
\Phi\equiv1\quad\text{on }(-\infty,R],
\qquad
\Phi\equiv0\quad\text{on }[(1+\varepsilon)R,\infty).
\]
Moreover, $R>R_*$ by the choice of $R_\flat$. Thus all the
hypotheses of Proposition~\ref{prop:cutoff} are satisfied,
with $R_1=R$, $R_2=(1+\varepsilon)R$, and $\Phi_-=1$.
Since $\chi^j=\Phi\circ r$, that proposition gives
\[
\begin{aligned}
-\int_M\lambda\langle\nabla(\chi^j),T\rangle
\le{}&
-n\int_M H\lambda N(\chi^j)
-\int_M\lambda\Bigl[
\Phi''(r)|\gl r|^2+nC_1\Phi'(r)
\Bigr].
\end{aligned}
\]
Consequently,
\begin{equation}\label{eq:flux-intermediate}
\begin{aligned}
\int_M\chi^j\operatorname{div}V
\le{}&
-\int_M(\kappa+nH\lambda)N(\chi^j)-\int_M\lambda\Phi''(r)|\gl r|^2
-nC_1\int_M\lambda\Phi'(r).
\end{aligned}
\end{equation}
Moreover, differentiating $\Phi$ gives
\[
\Phi'(t)
=\frac{j}{\varepsilon R}
\psi\left(\frac{t-R}{\varepsilon R}\right)^{j-1}
\psi'\left(\frac{t-R}{\varepsilon R}\right)
\]
and
\[
\begin{aligned}
\Phi''(t)
={}&\frac{j(j-1)}{(\varepsilon R)^2}
\psi\left(\frac{t-R}{\varepsilon R}\right)^{j-2}
\psi'\left(\frac{t-R}{\varepsilon R}\right)^2+\frac{j}{(\varepsilon R)^2}
\psi\left(\frac{t-R}{\varepsilon R}\right)^{j-1}
\psi''\left(\frac{t-R}{\varepsilon R}\right),
\end{aligned}
\]
where the first term is omitted when $j=1$.

Because $\psi'\le0$ and $|\psi'|\le C_\psi$, we have
\[
|\Phi'(r)|
\le\frac{jC_\psi}{\varepsilon R}\chi^{j-1}.
\]
The same formula shows that $\Phi'(r)=0$ almost everywhere
outside $A_R$. Since $|\nabla r|\le1$ almost everywhere,
\[
|N(\chi^j)|
=j\chi^{j-1}|N(\chi)|
\le\frac{jC_\psi}{\varepsilon R}\chi^{j-1}
\quad\text{almost everywhere on }A_R,
\]
and $N(\chi^j)=0$ almost everywhere outside $A_R$.

Next, the first term in the formula for $\Phi''$ is
nonnegative, because $j\ge2$ whenever this term is present.
It follows that
\[
-\Phi''(r)
\le\frac{jC'_\psi}{(\varepsilon R)^2}\chi^{j-1}.
\]
Since $|\gl r|\le|\nabla r|\le1$ almost everywhere, we obtain
\[
-\Phi''(r)|\gl r|^2
\le\frac{jC'_\psi}{(\varepsilon R)^2}\chi^{j-1}
\quad\text{almost everywhere on }A_R.
\]
Furthermore, $\Phi''(r)=0$ almost everywhere outside $A_R$.
Finally, since $C_1\ge0$ and $\Phi'\le0$,
\[
-nC_1\Phi'(r)
\le\frac{jnC_1C_\psi}{\varepsilon R}\chi^{j-1},
\]
and this term also vanishes almost everywhere outside $A_R$.

By the definitions of $\kappa$ and $\lambda$,
\[
\kappa+nH\lambda
=G^s+2sH^2G^{s-1}\ge0.
\]
Using the preceding estimates in
\eqref{eq:flux-intermediate}, together with $\lambda\ge0$,
we conclude that
\[
\begin{aligned}
\int_M\chi^j\operatorname{div}V
\le{}&
\frac{jC_\psi}{\varepsilon R}
\int_{A_R}\chi^{j-1}
\bigl(\kappa+nH\lambda\bigr)+\int_{A_R}\lambda\chi^{j-1}
\left(
\frac{jC'_\psi}{(\varepsilon R)^2}
+\frac{jnC_1C_\psi}{\varepsilon R}
\right)\\
={}&\frac{jC_\psi}{\varepsilon R}
\int_{A_R}\chi^{j-1}
\left[
G^s+2sH^2G^{s-1}
+\lambda\left(
nC_1+\frac{C'_\psi}{C_\psi\varepsilon R}
\right)
\right].
\end{aligned}
\]
Recalling the definition
$
C'_R=nC_1+\frac{C'_\psi}{C_\psi\varepsilon R}
$
and substituting $\lambda=\frac{2s}{n}HG^{s-1}$, we obtain
\[
\int_M\chi^j\operatorname{div}V
\le\frac{jC_\psi}{\varepsilon R}
\int_{A_R}\chi^{j-1}
\left[
G^s+2sH^2G^{s-1}
+\frac{2s}{n}C'_RHG^{s-1}
\right].
\]
The expression in brackets is precisely $\mathcal B_R$.
Finally, since $\operatorname{div}V\ge Y$ and $\chi^j\ge0$,
\[
I_j=\int_M\chi^jY
\le\int_M\chi^j\operatorname{div}V.
\]
Combining the last two inequalities proves
\[
I_j\le
\frac{jC_\psi}{\varepsilon R}
\int_{A_R}\chi^{j-1}\mathcal B_R,
\]
as claimed.
\end{proof}

\begin{lemma}[top-order bound]\label{lem:top}
On $M$, $\cB_R\le\theta'_RY$, where
\begin{equation}\label{eq:theta}
\theta'_R=\theta'(s,c_0,C'_R):=\max\Big\{\frac{2s+1}{(2s-n)c_0},\ \frac{(2s+1)c_0+\frac{2s}{n}C'_R}{2s\,\delta_0^2}\Big\}.
\end{equation}
\end{lemma}

\begin{proof}
Recall that
$
\mathcal B_R
=G^s+2sH^2G^{s-1}
+\frac{2s}{n}C'_RHG^{s-1},
$
whereas
$Y=(2s-n)HG^s+2s\delta_0^2HG^{s-1}$. Both $\mathcal B_R$ and $Y$ vanish on $M\setminus U$, where
$U=\{H>c_0\}$. Thus, it suffices to prove the desired
inequality on $U$.

Fix $x\in U$. Since $H(x)>c_0>0$, we have
$HG^{s-1}>0$. Dividing the inequality
$\mathcal B_R\le\theta'_R Y$ by this positive quantity, we
see that it is equivalent to
\[
\frac{G}{H}+2sH+\frac{2s}{n}C'_R
\le
\theta'_R\bigl((2s-n)G+2s\delta_0^2\bigr).
\]
On $U$, the identity $G=H^2-c_0^2$ gives
$
0<G<H^2.
$
Consequently,
\[
\frac{G}{H}\le H.
\]
Moreover, since $H>0$ and $G=H^2-c_0^2$, we have
\[
H=\frac{G+c_0^2}{H}
=\frac{G}{H}+\frac{c_0^2}{H}
\le \frac{G}{c_0}+c_0,
\]
where we used $H\ge c_0$ in the last inequality. It follows
that
\[
\begin{aligned}
\frac{G}{H}+2sH+\frac{2s}{n}C'_R
&\le (2s+1)H+\frac{2s}{n}C'_R\\
&\le \frac{2s+1}{c_0}G
+(2s+1)c_0+\frac{2s}{n}C'_R.
\end{aligned}
\]
By the definition of $\theta'_R$,
\[
\theta'_R\ge
\frac{2s+1}{(2s-n)c_0}
\quad\text{and}\quad
\theta'_R\ge
\frac{(2s+1)c_0+\frac{2s}{n}C'_R}
{2s\delta_0^2}.
\]
Since $2s>n$, both $(2s-n)G$ and $2s\delta_0^2$ are
nonnegative. Hence
\[
\frac{2s+1}{c_0}G
\le \theta'_R(2s-n)G
\]
and
\[
(2s+1)c_0+\frac{2s}{n}C'_R
\le \theta'_R\,2s\delta_0^2.
\]
Adding these inequalities yields
\[
\frac{G}{H}+2sH+\frac{2s}{n}C'_R
\le\theta'_R\bigl((2s-n)G+2s\delta_0^2\bigr).
\]
Multiplying by $HG^{s-1}>0$, we conclude that
$
\mathcal B_R\le\theta'_R Y
$
on $U$.
Together with the vanishing of both sides on $M\setminus U$,
this proves the lemma.
\end{proof}

\begin{corollary}\label{cor:iter}
For $R\ge R_0\ge R_\flat$ and $j\ge1$, $I_j\le\dfrac{jC_\psi\theta'_{R_0}}{\varepsilon R}\,I_{j-1}$.
\end{corollary}
\begin{proof}
By Proposition~\ref{prop:flux} and Lemma~\ref{lem:top}, we obtain
\[
I_j
\le
\frac{jC_\psi\theta'_R}{\varepsilon R}
\int_{A_R}\chi^{j-1}Y.
\]
Since $C'_R$ is nonincreasing in $R$ and $\theta'_R$ is
nondecreasing as a function of $C'_R$, the condition
$R\ge R_0$ implies
$
\theta'_R\le\theta'_{R_0}.
$
Therefore,
\[
I_j
\le
\frac{jC_\psi\theta'_{R_0}}{\varepsilon R}
\int_{A_R}\chi^{j-1}Y.
\]

If $j\ge2$, then $0\le\chi\le1$ and $Y\ge0$, so
\[
\int_{A_R}\chi^{j-1}Y
\le
\int_M\chi^{j-1}Y
=I_{j-1}.
\]
If $j=1$, the definition of $I_0$ and the inclusion
\[
A_R\subset\{r\le(1+\varepsilon)R\}
\]
give
\[
\int_{A_R}Y
\le
\int_{\{r\le(1+\varepsilon)R\}}Y
=I_0.
\]
Combining these estimates yields
\[
I_j\le
\frac{jC_\psi\theta'_{R_0}}{\varepsilon R}\,I_{j-1}
\]
for every $j\ge1$, as claimed.
\end{proof}

\begin{lemma}[a priori bound]\label{lem:apriori}
There is a constant $C_2$, depending only on $n,s,c_0,K_0,\psi,\varepsilon$ and $R_*$, such that $J(R)\le C_2\,\mathcal V(R)$ for all $R\ge R_\flat$.
\end{lemma}

\begin{proof}
Set $q:=2s+1$ and fix $R\ge R_\flat$. By
Proposition~\ref{prop:flux}, applied with $j=q$, we have
\[
I_q\le
\frac{qC_\psi}{\varepsilon R}
\int_{A_R}\chi^{q-1}\cB_R.
\]
Recall that $\cB_R=0$ on $M\setminus U$, where
$U=\{H>c_0\}$. We split $U=U_1\cup U_2$, where
\[
U_1:=\{c_0<H<2c_0\},
\qquad
U_2:=\{H\ge2c_0\}.
\]

On $U_1$, we have $G<3c_0^2$ and $H<2c_0$.
Moreover, $C'_R\le C'_{R_\flat}$. Consequently, there
exists a constant $B_0>0$, depending only on
$n,s,c_0,K_0,\psi,\varepsilon$ and $R_*$, such that
\[
\cB_R\le B_0\qquad\text{on }U_1.
\]

On $U_2$, the inequalities
\[
G=H^2-c_0^2\ge\frac34H^2,
\qquad
F=HG^s\le H^{2s+1}=H^q
\]
imply $H\ge F^{1/q}$. Hence
\[
G^s=\frac{F}{H}
\le F^{\frac{q-1}{q}},
\]
and
\[
H^2G^{s-1}
=F\frac{H}{G}
\le\frac{4F}{3H}
\le\frac43F^{\frac{q-1}{q}}.
\]
Similarly,
\[
HG^{s-1}
=\frac{F}{G}
\le\frac{4F}{3H^2}
\le\frac43F^{\frac{q-2}{q}}.
\]
It follows that
\[
\cB_R\le
B_1\left(
F^{\frac{q-1}{q}}+
F^{\frac{q-2}{q}}
\right)
\qquad\text{on }U_2,
\]
where
$
B_1:=
1+\frac{8s}{3}
+\frac{8s}{3n}C'_{R_\flat}.
$

Define $X:=\chi^qF$. Since $0\le\chi\le1$, we obtain
\[
\chi^{q-1}F^{\frac{q-1}{q}}
=X^{\frac{q-1}{q}},
\qquad
\chi^{q-1}F^{\frac{q-2}{q}}
\le X^{\frac{q-2}{q}}.
\]
Combining the estimates on $U_1$ and $U_2$, we conclude
that
\[
\begin{aligned}
I_q
&\le \frac{qC_\psi B_0}{\varepsilon R}
       \operatorname{vol}(A_R)+
m\int_{A_R}
\left(
X^{\frac{q-1}{q}}+
X^{\frac{q-2}{q}}
\right),
\end{aligned}
\]
where
$
m:=\frac{qC_\psi B_1}{\varepsilon R}
\le
m_1:=\frac{qC_\psi B_1}{\varepsilon R_\flat}.
$
For every $0<\vartheta<1$, Young's inequality gives
\[
mX^\vartheta
\le \eta X+
C(\eta,\vartheta)m^{\frac{1}{1-\vartheta}},
\]
for any $\eta>0$, where $C(\eta,\vartheta)$ depends
only on $\eta$ and $\vartheta$. Applying this inequality
with $\vartheta=(q-1)/q$ and $\vartheta=(q-2)/q$, we find
\[
\begin{aligned}
I_q
&\le
\left[
\frac{qC_\psi B_0}{\varepsilon R_\flat}
+C_\eta\left(m_1^q+m_1^{q/2}\right)
\right]\operatorname{vol}(A_R)+2\eta\int_M X,
\end{aligned}
\]
where $C_\eta$ depends only on $\eta$ and $q$. Since
$
Y=(2s-n)F+2s\delta_0^2HG^{s-1}
\ge(2s-n)F
$
and $X=\chi^qF$, we have
\[
\int_M X\le\frac{I_q}{2s-n}.
\]
Choose $\eta=(2s-n)/4>0$. Then
\[
I_q\le C\,\operatorname{vol}(A_R)+\frac12 I_q,
\]
where $C$ depends only on
$n,s,c_0,K_0,\psi,\varepsilon$ and $R_*$.
Absorbing the last term into the left-hand side yields
\[
I_q\le 2C\,\operatorname{vol}(A_R)
\le 2C\,\operatorname{vol}
\{r\le(1+\varepsilon)R\}.
\]
Finally, since $\chi=1$ on $\{r\le R\}$ and $Y\ge0$,
\[
J(R)=\int_{\{r\le R\}}Y
\le\int_M\chi^qY=I_q.
\]
Using the volume estimate \eqref{eq:BG}, we conclude that
\[
J(R)\le C_2\,\mathcal V(R),
\]
after enlarging the constant $C_2$ if necessary.
This proves the lemma.
\end{proof}

\subsection{Conclusion of the proof}

For $n\ge1$ and $K_0,C_1\ge0$, define
\begin{equation}\label{eq:cstar}
c_*(n,K_0;C_1):=
\frac12\left[
n\sqrt{K_0}
+\sqrt{(n^2+4)K_0+4n\sqrt{K_0}\,C_1}
\right].
\end{equation}
This is the nonnegative root of
\[
c^2-n\sqrt{K_0}\,c-K_0
-n\sqrt{K_0}\,C_1=0.
\]
The function $c_*$ is continuous and nondecreasing in $C_1$.
Moreover, since $n\ge1$ and $C_1\ge0$, we have
\[
c_*(n,K_0;C_1)
\ge
\frac12\left(n+\sqrt{n^2+4}\right)\sqrt{K_0}
\ge \sqrt{K_0}.
\]
When $K_0=0$, the defining formula gives
\[
c_*(n,0;C_1)=0.
\]
Finally, taking $C_1=\sqrt{K_0}$ in \eqref{eq:cstar}, we obtain
\begin{equation}\label{eq:cstar-explicit}
\begin{aligned}
c_*(n,K_0;\sqrt{K_0})
&=\frac12\left[
n\sqrt{K_0}
+\sqrt{(n^2+4+4n)K_0}
\right]\\
&=(n+1)\sqrt{K_0}.
\end{aligned}
\end{equation}

\begin{proposition}\label{prop:main}
For every $R_*>0$, $|H|\le c_*\big(n,K_0;\ct_{K_0}(R_*)\big)$ on $M$.
\end{proposition}

\begin{proof}We divide the proof into four steps. Fix $R_*>0$ and set $C_1:=\ct_{K_0}(R_*)$, $c_*:=c_*(n,K_0;C_1)$.
We prove first that $H\le c_*$ on $M$. The opposite
inequality will follow by reversing the orientation of the
foliation.

\medskip
\noindent
\textit{Step 1: Forced exponential growth.}
Fix $R_0\ge R_\flat$ and set
$
\theta:=C_\psi\theta'_{R_0}>0.
$
For every
$
R\ge\max\left\{R_0,\frac{\theta}{\varepsilon}\right\},
$
let
$
q:=\left\lfloor\frac{\varepsilon R}{\theta}\right\rfloor
\ge1.
$
By iterating Corollary~\ref{cor:iter}, we obtain
\[
I_q
\le
q!\left(\frac{C_\psi\theta'_{R_0}}
{\varepsilon R}\right)^q I_0
=
q!\left(\frac{\theta}{\varepsilon R}\right)^q I_0.
\]
Since $q\theta\le\varepsilon R$, it follows that
\[
I_q\le\frac{q!}{q^q}I_0.
\]
By \eqref{eq:Imono} and the definition of $I_0$,
\[
J(R)\le I_q,
\qquad
I_0=J((1+\varepsilon)R).
\]
Consequently,
\[
J(R)\le\frac{q!}{q^q}J((1+\varepsilon)R).
\]
Using Stirling's estimate
\[
q!\le e\sqrt q\left(\frac qe\right)^q,
\]
we obtain
\[
J(R)\le e\sqrt q\,e^{-q}J((1+\varepsilon)R).
\]
Moreover, the definition of $q$ gives
\[
\frac{\varepsilon R}{\theta}-1
<q\le\frac{\varepsilon R}{\theta}.
\]
In particular,
\[
\sqrt q\le\sqrt{\frac{\varepsilon R}{\theta}}
\quad\text{and}\quad
e^{-q}<e\exp\left(-\frac{\varepsilon R}{\theta}\right).
\]
Thus
\begin{equation}\label{eq:step1}
J(R)\le
e^2\sqrt{\frac{\varepsilon R}{\theta}}\,
\exp\left(-\frac{\varepsilon R}{\theta}\right)
J((1+\varepsilon)R)
\end{equation}
for every
$R\ge\max\{R_0,\theta/\varepsilon\}$.

\medskip
\noindent
\textit{Step 2: Comparison with the volume growth.}
Suppose that
\begin{equation}\label{eq:rate}
n\sqrt{K_0}(1+\varepsilon)\theta<1.
\end{equation}
We show that this condition forces $J\equiv0$.

Fix
$
\bar R\ge\max\left\{R_0,\frac{\theta}{\varepsilon}\right\}
$
and define
\[
\varrho_i:=(1+\varepsilon)^i\bar R,
\qquad i=0,1,2,\ldots.
\]
Applying \eqref{eq:step1} successively at
$R=\varrho_0,\ldots,\varrho_{k-1}$, we find
\[
J(\bar R)
\le
\left[
\prod_{i=0}^{k-1}
e^2\sqrt{\frac{\varepsilon\varrho_i}{\theta}}
\right]
\exp\left(
-\frac{\varepsilon}{\theta}
\sum_{i=0}^{k-1}\varrho_i
\right)
J(\varrho_k).
\]
The geometric-series identity gives
\[
\sum_{i=0}^{k-1}\varrho_i
=\frac{\varrho_k-\bar R}{\varepsilon}.
\]
Furthermore, $\varrho_i\le\varrho_k$ for $i<k$, and hence
\[
\prod_{i=0}^{k-1}
e^2\sqrt{\frac{\varepsilon\varrho_i}{\theta}}
\le
\left(e^2\sqrt{\frac{\varepsilon\varrho_k}{\theta}}\right)^k.
\]
It follows that
\[
J(\bar R)\le
\left(e^2\sqrt{\frac{\varepsilon\varrho_k}{\theta}}\right)^k
\exp\left(-\frac{\varrho_k-\bar R}{\theta}\right)
J(\varrho_k).
\]

By Lemma~\ref{lem:apriori} and the volume estimate
\eqref{eq:BG},
\[
J(\varrho_k)
\le
C_2\omega_n
\big((1+\varepsilon)\varrho_k\big)^{n+1}
\exp\left(
n\sqrt{K_0}(1+\varepsilon)\varrho_k
\right).
\]
Therefore,
\begin{align*}
J(\bar R)
&\le
\left(e^2\sqrt{\frac{\varepsilon\varrho_k}{\theta}}\right)^k
\exp\left(-\frac{\varrho_k-\bar R}{\theta}\right) 
C_2\omega_n
\big((1+\varepsilon)\varrho_k\big)^{n+1}
\exp\left(
n\sqrt{K_0}(1+\varepsilon)\varrho_k
\right).
\end{align*}

Since
$
k=\frac{\log(\varrho_k/\bar R)}
{\log(1+\varepsilon)},
$
we have $k=O(\log\varrho_k)$ as $k\to\infty$.
Taking logarithms of the positive expression on the
right-hand side above, its logarithm is
\[
-\left[
\frac1\theta-n\sqrt{K_0}(1+\varepsilon)
\right]\varrho_k
+O\big((\log\varrho_k)^2\big).
\]
Indeed, the product of the $k$ factors contributes
$O(k\log\varrho_k)=O((\log\varrho_k)^2)$, while the
remaining polynomial factors contribute only
$O(\log\varrho_k)$.

By \eqref{eq:rate}, the coefficient
\[
\frac1\theta-n\sqrt{K_0}(1+\varepsilon)
\]
is strictly positive. Since
\[
\frac{(\log\varrho_k)^2}{\varrho_k}\longrightarrow0,
\]
the right-hand side of the estimate for $J(\bar R)$
tends to zero as $k\to\infty$. Hence
$
J(\bar R)=0.
$
As $\bar R$ was arbitrary above the fixed threshold,
$J(R)=0$ for all sufficiently large $R$. Because $J$ is
nonnegative and nondecreasing, this implies
\[
J(R)=0\qquad\text{for every }R>0.
\]
Thus $J\equiv0$. Since $Y\ge0$ is continuous and
\[
J(R)=\int_{\{r\le R\}}Y,
\]
we conclude that $Y\equiv0$ on $M$.

Recall that
\[
Y=(2s-n)HG^s+2s\delta_0^2HG^{s-1},
\qquad
\delta_0^2=c_0^2-K_0>0.
\]
On $U=\{H>c_0\}$, we have $H>0$ and $G>0$.
Consequently,
\[
Y\ge 2s\delta_0^2HG^{s-1}>0
\qquad\text{on }U,
\]
contradicting $Y\equiv0$ unless $U=\varnothing$.
Therefore,
\[
H\le c_0\qquad\text{on }M,
\]
provided the parameters satisfy \eqref{eq:rate}.

\medskip
\noindent
\textit{Step 3: Choice of the parameters.}
Let $c_0>c_*$. By the properties of $c_*$,
\[
c_0>\sqrt{K_0},
\qquad
\delta_0^2:=c_0^2-K_0>0.
\]
We show that the parameters in the preceding steps can
be chosen so that \eqref{eq:rate} holds.

If $K_0=0$, then \eqref{eq:rate} holds automatically,
since its left-hand side is zero. We may therefore assume
that $K_0>0$.

By the definition of $c_*$, the inequality $c_0>c_*$ is
equivalent to
\[
c_0^2-n\sqrt{K_0}\,c_0-K_0
-n\sqrt{K_0}\,C_1>0.
\]
Equivalently,
\begin{equation}\label{eq:cond}
n\sqrt{K_0}\,
\frac{c_0+C_1}{c_0^2-K_0}<1.
\end{equation}
Moreover,
\[
\frac1{c_0}
\le
\frac{c_0+C_1}{c_0^2-K_0},
\]
because this inequality is equivalent to
\[
c_0(c_0+C_1)-(c_0^2-K_0)
=c_0C_1+K_0\ge0.
\]

Recall that
\[
\theta'(s,c_0,C'_R)
=
\max\left\{
\frac{2s+1}{(2s-n)c_0},
\frac{(2s+1)c_0+\frac{2s}{n}C'_R}
{2s(c_0^2-K_0)}
\right\}.
\]
For fixed $\sigma>0$, as $s\to\infty$, the two
entries in this maximum, with $C'_R=nC_1+\sigma$,
converge respectively to
\[
\frac1{c_0}
\quad\text{and}\quad
\frac{c_0+C_1+\sigma/n}{c_0^2-K_0}.
\]
By \eqref{eq:cond}, we can first choose $\sigma>0$
sufficiently small so that
\[
n\sqrt{K_0}\,
\frac{c_0+C_1+\sigma/n}{c_0^2-K_0}<1.
\]
Together with
\[
\frac1{c_0}
\le\frac{c_0+C_1}{c_0^2-K_0},
\]
this ensures that both limiting entries satisfy the
strict inequality required by the rate condition.

It follows that we can choose $\tau>0$ sufficiently
small, $\varepsilon\in(0,1]$ sufficiently small, and
an integer $s$ sufficiently large with $s\geq 3$ and $2s>n$, such that
\begin{equation}\label{eq:rate-parameters}
n\sqrt{K_0}(1+\varepsilon)(1+\tau)
\theta'(s,c_0,nC_1+\sigma)<1.
\end{equation}
Choose a cutoff function $\psi$ with
$C_\psi\le1+\tau$. Finally, by the definition
\[
C'_R=nC_1+\frac{C'_\psi}{C_\psi\varepsilon R},
\]
we can choose $R_0\ge R_\flat$ sufficiently large that
\[
C'_{R_0}\le nC_1+\sigma.
\]
Since $\theta'$ is nondecreasing in its last argument,
\[
\theta'_{R_0}
=\theta'(s,c_0,C'_{R_0})
\le\theta'(s,c_0,nC_1+\sigma).
\]
Thus, with $\theta=C_\psi\theta'_{R_0}$, we obtain
\[
n\sqrt{K_0}(1+\varepsilon)\theta
\le
n\sqrt{K_0}(1+\varepsilon)(1+\tau)
\theta'(s,c_0,nC_1+\sigma)<1.
\]
Hence \eqref{eq:rate} holds. Step~2 now yields
\[
H\le c_0\qquad\text{on }M.
\]
Since this conclusion holds for every $c_0>c_*$,
letting $c_0\downarrow c_*$ gives
\[
H\le c_*\qquad\text{on }M.
\]

\medskip
\noindent
\textit{Step 4: The opposite inequality.}
Reverse the orientation of the foliation by replacing
$N$ with $-N$. The corresponding mean curvature becomes
$-H$, while
\[
\nabla_{-N}(-N)=\nabla_NN=T.
\]
Thus the vector field $T$, the underlying foliation,
and the sectional-curvature assumption remain unchanged.
Applying the upper bound just proved to the reversed
orientation gives
\[
-H\le c_*\qquad\text{on }M.
\]
Combining this with $H\le c_*$, we conclude that
\[
|H|\le c_*
=c_*\bigl(n,K_0;\ct_{K_0}(R_*)\bigr)
\qquad\text{on }M.
\]
This completes the proof.
\end{proof}

\begin{proof}[Proof of Theorem A]
By Proposition~\ref{prop:main}, for every $R_*>0$,
\[
|H|\le c_*\bigl(n,K_0;\ct_{K_0}(R_*)\bigr)
\qquad\text{on }M.
\]
Since
\[
\lim_{R_*\to\infty}\ct_{K_0}(R_*)=\sqrt{K_0},
\]
the continuity of $c_*(n,K_0;\,\cdot\,)$ implies that
\[
\lim_{R_*\to\infty}
c_*\bigl(n,K_0;\ct_{K_0}(R_*)\bigr)
=c_*(n,K_0;\sqrt{K_0}).
\]
Passing to the limit in the estimate above, we obtain
\[
|H|\le c_*(n,K_0;\sqrt{K_0}).
\]
By \eqref{eq:cstar-explicit},
$
c_*(n,K_0;\sqrt{K_0})=(n+1)\sqrt{K_0}.
$
Therefore,
\[
|H|\le(n+1)\sqrt{K_0}
\qquad\text{on }M.
\]
In particular, if $K_0=0$, then $|H|\le0$, and hence
$H\equiv0$ on $M$. Thus every leaf of the foliation is
minimal.
\end{proof}

\appendix
\section{Proofs of Lemmas~\ref{lem:1d} and~\ref{lem:reg}}\label{app:reg}

\subsection{Proof of Lemma~\ref{lem:1d}}

\begin{customlemma}{\ref*{lem:1d}}
Let $I\subset\mathbb{R}$ be an open interval, and let
$v,\beta\in C^0(I)$. Assume that for every $t_0\in I$ and every
$\varepsilon>0$, there exists a $C^2$ function $\phi$ defined on a
neighborhood of $t_0$ such that
\[
\phi\geq v,\qquad \phi(t_0)=v(t_0),\qquad
\phi''(t_0)\leq\beta(t_0)+\varepsilon.
\]
Then $v''\leq\beta$ in the sense of distributions on $I$; that is,
\[
\int_I v(t)\zeta''(t)\,dt
\leq
\int_I \beta(t)\zeta(t)\,dt
\]
for every nonnegative $\zeta\in C_c^\infty(I)$.
\end{customlemma}

\begin{proof}
Since $\beta$ is continuous on $I$, there exists a function
$B\in C^2(I)$ such that $B''=\beta$. Set $w:=v-B.$ We first observe that $w$ satisfies the following upper support
property: for every $t_0\in I$ and every $\varepsilon>0$, there
exists a $C^2$ function $\varphi$ defined on a neighborhood of
$t_0$ such that
\[
\varphi\geq w,\qquad
\varphi(t_0)=w(t_0),\qquad
\varphi''(t_0)\leq\varepsilon.
\]
Indeed, given $t_0\in I$ and $\varepsilon>0$, apply the hypothesis
to $v$ with the same $\varepsilon$. Let $\phi_0$ be the resulting
upper support function, so that
\[
\phi_0\geq v,\qquad
\phi_0(t_0)=v(t_0),\qquad
\phi_0''(t_0)\leq\beta(t_0)+\varepsilon.
\]
Define $\varphi:=\phi_0-B$. Then $\varphi\geq v-B=w$, $\varphi(t_0)=w(t_0)$, and
\[
\begin{aligned}
\varphi''(t_0)
&=\phi_0''(t_0)-B''(t_0)\\
&\leq\beta(t_0)+\varepsilon-\beta(t_0)
=\varepsilon.
\end{aligned}
\]
This establishes the asserted upper support property of $w$.

We claim that $w$ is concave on $I$. Suppose, to the contrary,
that $w$ is not concave. Then there exist $a,b\in I$, with $a<b$,
such that $w$ lies strictly below the chord joining
$(a,w(a))$ and $(b,w(b))$ at some point of $(a,b)$. Let $\ell$
be the affine function satisfying $\ell(a)=w(a)$, $\ell(b)=w(b)$. Since $w$ is not concave, we have
\[
m:=\max_{t\in[a,b]}\bigl(\ell(t)-w(t)\bigr)>0.
\]
Choose $\delta>0$ sufficiently small that $\frac{\delta(b-a)^2}{4}<m$. Define
\[
\tilde{h}(t):=\ell(t)-\delta(t-a)(b-t)-w(t),
\]
$t\in[a,b]$. Since $w$ and $\ell$ are continuous, $\tilde{h}$ is continuous on
$[a,b]$. Moreover, $\tilde{h}(a)=\tilde{h}(b)=0$. Since $0\leq(t-a)(b-t)\leq\frac{(b-a)^2}{4}$ for all $t\in[a,b]$, we obtain
\[
\max_{t\in[a,b]}\tilde{h}(t)
\geq m-\frac{\delta(b-a)^2}{4}>0.
\]
Consequently, $h$ attains its positive maximum at some
$t^*\in(a,b)$.

Now define the $C^2$ function
\[
\psi(t):=\ell(t)-\delta(t-a)(b-t)-\tilde{h}(t^*).
\]
By the definition of $h$ and the maximality of $\tilde{h}(t^*)$, we have
\[
\begin{aligned}
\psi(t)-w(t)
&=\ell(t)-\delta(t-a)(b-t)-\tilde{h}(t^*)-w(t)\\
&=\tilde{h}(t)-\tilde{h}(t^*)\leq0
\end{aligned}
\]
for every $t\in[a,b]$. Hence $\psi\leq w$ on $[a,b]$, $\psi(t^*)=w(t^*)$. Furthermore, since $\ell$ is affine,
$\psi''(t)=2\delta$ for every $t\in[a,b]$.

We now apply the upper support property of $w$ at $t^*$ with
$\varepsilon=\delta$. There exists a $C^2$ function $\varphi_1$
defined on a neighborhood of $t^*$ such that
\[
\varphi_1\geq w,\qquad
\varphi_1(t^*)=w(t^*),\qquad
\varphi_1''(t^*)\leq\delta.
\]
Since $\psi\leq w$ on $[a,b]$ and $t^*\in(a,b)$, we have
$\varphi_1\geq\psi$ on a neighborhood of $t^*$, with equality
at $t^*$. Thus $\varphi_1-\psi$ has a local minimum at $t^*$,
and therefore
\[
(\varphi_1-\psi)''(t^*)\geq0.
\]
It follows that
\[
\varphi_1''(t^*)\geq\psi''(t^*)=2\delta,
\]
contradicting $\varphi_1''(t^*)\leq\delta$, since $\delta>0$.
Therefore, $w$ is concave on $I$.

It remains to show that the concavity of $w$ implies
\[
\int_I w(t)\zeta''(t)\,dt\leq0
\]
for every nonnegative $\zeta\in C_c^\infty(I)$. Fix such a function $\zeta$. For every sufficiently small $h>0$,
the concavity of $w$ gives
\[
w(t+h)-2w(t)+w(t-h)\leq0
\]
whenever $t-h,t,t+h\in I$. Since $\zeta$ has compact support,
we may choose $h>0$ sufficiently small so that all relevant
points lie in $I$. Multiplying by $\zeta(t)\geq0$ and integrating,
we obtain
\[
\int_I
\zeta(t)\frac{w(t+h)-2w(t)+w(t-h)}{h^2}\,dt
\leq0.
\]
A change of variables in the first and third terms, together
with the compact support of $\zeta$, yields
\[
\int_I
w(t)\frac{\zeta(t+h)-2\zeta(t)+\zeta(t-h)}{h^2}\,dt
\leq0.
\]
As $h\to0$, the central second difference of $\zeta$ converges
uniformly to $\zeta''$ on a compact subset of $I$ containing the
supports of these functions. Since $w$ is continuous, we may pass
to the limit and conclude that
\[
\int_I w(t)\zeta''(t)\,dt\leq0.
\]

Finally, recalling that $v=w+B$, we obtain
\[
\begin{aligned}
\int_I v(t)\zeta''(t)\,dt
&=\int_I w(t)\zeta''(t)\,dt
  +\int_I B(t)\zeta''(t)\,dt\\
&\leq\int_I B(t)\zeta''(t)\,dt.
\end{aligned}
\]
Since $\zeta$ has compact support, integration by parts twice
gives
\[
\int_I B(t)\zeta''(t)\,dt
=\int_I B''(t)\zeta(t)\,dt
=\int_I\beta(t)\zeta(t)\,dt.
\]
Therefore,
\[
\int_I v(t)\zeta''(t)\,dt
\leq\int_I\beta(t)\zeta(t)\,dt
\]
for every nonnegative $\zeta\in C_c^\infty(I)$, as desired.
\end{proof}

\subsection{Proof of Lemma~\ref{lem:reg}}\label{subapp:reg}

\begin{lemma}\label{lem:supergrad}
Let $f$ be $L$-Lipschitz on a neighbourhood of $x$ and let $\varphi$ be $C^1$ near $x$ with $\varphi\ge f$ near $x$ and $\varphi(x)=f(x)$. Then $|\nabla\varphi(x)|\le L$.
\end{lemma}
\begin{proof}
Let $v\in T_xM$ be a unit vector. For all sufficiently small
$t>0$, the minimizing property of radial geodesics gives $\operatorname{dist}_M\bigl(x,\exp_x(tv)\bigr)=t.$
Since $f$ is $L$-Lipschitz, we have $|f(y)-f(x)|\leq L\,\operatorname{dist}_M(x,y)$
for all $y$ sufficiently close to $x$. In particular, $f(y)\geq f(x)-L\,\operatorname{dist}_M(x,y).$ Taking $y=\exp_x(tv)$ and using $\varphi\geq f$ and
$\varphi(x)=f(x)$, we obtain
\[
\begin{aligned}
\varphi(\exp_x(tv))
&\geq f(\exp_x(tv))\\
&\geq f(x)-L\,\operatorname{dist}_M\bigl(x,\exp_x(tv)\bigr)\\
&=f(x)-Lt\\
&=\varphi(x)-Lt.
\end{aligned}
\]
Thus,
\[
\frac{\varphi(\exp_x(tv))-\varphi(x)}{t}\geq -L.
\]
Letting $t\to0^+$ and using the differentiability of $\varphi$,
we conclude that
\[
\langle\nabla\varphi(x),v\rangle\geq-L.
\]
Applying the same argument to $-v$ yields $\langle\nabla\varphi(x),v\rangle\leq L.$ Therefore, $|\langle\nabla\varphi(x),v\rangle|\leq L$ for every unit vector $v\in T_xM$. Taking the supremum over all unit vectors, we obtain
$|\nabla\varphi(x)|\leq L.$ This completes the proof.
\end{proof}

\begin{customlemma}{\ref*{lem:reg}}
Let $W\subset M$ be open, $a\in\R$, $L>0$, and let $f\in C^0(W)$ be such that every point of $W$ has a neighbourhood on which $f$ is $L$-Lipschitz, and $\Hess f\le a\langle\cdot,\cdot\rangle$ in the support sense at every point of $W$. Let $K\subset W$ be compact and $\eta>0$. Then there are $\sigma>0$ and points $c_1,\dots,c_m\in K$ such that, with $U_i:=B_{2\sigma}(c_i)$:
\begin{enumerate}
\item[\textup{(i)}] $K\subset U_1\cup\dots\cup U_m$, and $\exp_{c_i}$ maps the Euclidean ball $B(0,4\sigma)\subset T_{c_i}M\cong\R^{n+1}$ (identified by an orthonormal basis) diffeomorphically onto the geodesic ball $B_{4\sigma}(c_i)\subset W$;
\item[\textup{(ii)}] for $0<\delta<\sigma$ the functions $f_{i,\delta}:=(u_i*\mu_\delta)\circ\exp_{c_i}^{-1}$, $u_i:=f\circ\exp_{c_i}$, are smooth on $B_{3\sigma}(c_i)$, $|\nabla f_{i,\delta}|\le2L$ there, $f_{i,\delta}\to f$ uniformly on $\overline U_i$ and $\nabla f_{i,\delta}\to\nabla f$ a.e. on $U_i$ as $\delta\to0$;
\item[\textup{(iii)}] there are $\omega_i(\delta)\to0$ such that $\Hess f_{i,\delta}\le(a+\eta+\omega_i(\delta))\langle\cdot,\cdot\rangle$ on $\overline U_i$.
\end{enumerate}
\end{customlemma}

\begin{proof}
We divide the proof into three steps, corresponding to assertions
\textup{(i)}, \textup{(ii)}, and \textup{(iii)}.

\medskip
\noindent
\textit{Proof of item \textup{(i)}.}

Since $K\subset W$ is compact and $W$ is open, there exists $r_W>0$ such that $B_{r_W}(c)\subset W$ for every $c\in K$. Moreover, the injectivity radius function
$c\mapsto\operatorname{inj}(c)$ is positive and continuous.
Consequently,
$
r_K:=\min_{c\in K}\operatorname{inj}(c)>0.
$
Thus, for every $c\in K$, the exponential map is a
diffeomorphism from $B(0,r)$ onto $B_r(c)$ whenever
$
0<r<\min\{r_W,r_K\}.
$

We next establish uniform estimates in normal coordinates.
The associated normal coordinate maps depend smoothly on
the base point and the coordinate variables. Since the metric
coefficients satisfy $g_{ij}(0)=\delta_{ij}$ at the origin
of each normal chart, compactness and smooth dependence imply
that, after decreasing $r>0$ if necessary,
\[
\frac12 I\leq (g_{ij}(y))\leq 2I
\]
uniformly for all base points $c\in K$ and all
$y\in B(0,r)$.

Furthermore, the Christoffel symbols vanish at the origin of normal coordinates and their first derivatives with respect to the coordinate
variables are uniformly bounded on these coordinate balls. Indeed, cover $K$ by finitely many coordinate
neighborhoods $V_1,\ldots,V_N$, each equipped with a smooth
orthonormal frame. For each $\alpha\in\{1,\ldots,N\}$, choose
a compact set $K_\alpha\subset V_\alpha$ such that
\[
K\subset\bigcup_{\alpha=1}^N K_\alpha.
\]
Let
$
F_\alpha(p,y)
:=
\exp_p\bigl(E_\alpha(p)y\bigr),
$
where $E_\alpha(p)\colon\mathbb R^{n+1}\to T_pM$
is the isometry determined by the chosen orthonormal frame.
Since the Riemannian metric is smooth, the maps $F_\alpha$
and the corresponding metric coefficients
$g_{ij}^{(\alpha)}(p,y)$ depend smoothly on $(p,y)$
wherever they are defined. Moreover, the matrices
$g^{(\alpha)}(p,y)$ are positive definite, so their inverses
also depend smoothly on $(p,y)$. It follows from the formula
\[
\Gamma^{(\alpha),k}_{ij}(p,y)
=
\frac12\sum_{\ell=1}^{n+1}
g^{(\alpha),k\ell}(p,y)
\left(
\partial_i g^{(\alpha)}_{j\ell}(p,y)
+\partial_j g^{(\alpha)}_{i\ell}(p,y)
-\partial_\ell g^{(\alpha)}_{ij}(p,y)
\right)
\]
that the Christoffel symbols depend smoothly on $(p,y)$.
In particular, the map
\[
(p,y)\longmapsto D_y\Gamma^{(\alpha)}(p,y)
\]
is continuous.

After decreasing the radius $r>0$, if necessary, the
coordinate maps are defined on a neighborhood of
$K_\alpha\times\overline{B(0,r)}$ for every $\alpha$.
Since each set $K_\alpha\times\overline{B(0,r)}$ is compact,
the continuity of $D_y\Gamma^{(\alpha)}$ implies that
there exists a constant $C_\alpha\geq0$ such that
\[
\bigl\|D_y\Gamma^{(\alpha)}(p,y)\bigr\|
\leq C_\alpha
\]
for every $p\in K_\alpha$ and
$y\in\overline{B(0,r)}$. Since the cover is finite, setting
$
C:=\max_{1\leq\alpha\leq N}C_\alpha
$
gives a uniform bound independent of the base point and the chosen coordinate neighborhood. Finally, since the Christoffel symbols vanish at the origin
of normal coordinates, we have
$
\Gamma^{(\alpha)}(p,0)=0.
$
Thus, by the fundamental theorem of calculus, for every
$p\in K_\alpha$ and $y\in B(0,r)$,
\[
\begin{aligned}
\Gamma^{(\alpha)}(p,y)
&=\Gamma^{(\alpha)}(p,y)
  -\Gamma^{(\alpha)}(p,0)\\
&=\int_0^1
D_y\Gamma^{(\alpha)}(p,ty)[y]\,dt.
\end{aligned}
\]
Consequently,
\[
\begin{aligned}
\bigl\|\Gamma^{(\alpha)}(p,y)\bigr\|
&\leq
\int_0^1
\bigl\|D_y\Gamma^{(\alpha)}(p,ty)\bigr\|
\,|y|\,dt\\
&\leq C|y|.
\end{aligned}
\]
Taking $\Lambda:=C$, we conclude that
\[
\bigl\|\Gamma^{(\alpha)}(p,y)\bigr\|
\leq\Lambda|y|
\]
uniformly for all base points $p\in K$ and all
$y\in B(0,r)$ in the corresponding normal coordinates. Hence there exists $\Lambda\geq0$, independent of $c\in K$,
such that
\[
\|\Gamma(y)\|\leq\Lambda |y|
\qquad\text{for all }y\in B(0,r).
\]

Choose $\sigma_0>0$ sufficiently small that
$4\sigma_0<r$. Then, for every $c\in K$,
\[
\exp_c:B(0,4\sigma_0)\longrightarrow B_{4\sigma_0}(c)
\]
is a diffeomorphism onto a ball contained in $W$, and
the estimates
\begin{equation}\label{eq:normal}
\frac12 I\leq(g_{ij}(y))\leq2I,
\qquad
\|\Gamma(y)\|\leq\Lambda|y|
\end{equation}
hold uniformly.

Now choose $0<\sigma\leq\sigma_0$ so small that
$
14\sqrt2\,\Lambda L\sigma\leq\eta.
$
If $\Lambda=0$, this restriction is automatic. By compactness
of $K$, there exist finitely many points
$c_1,\ldots,c_m\in K$ such that
\[
K\subset\bigcup_{i=1}^m B_\sigma(c_i).
\]
Set
$
U_i:=B_{2\sigma}(c_i),
$
for $i=1,\ldots,m.$
Then
\[
K\subset\bigcup_{i=1}^m U_i,
\]
and each exponential map
$
\exp_{c_i}:B(0,4\sigma)\longrightarrow B_{4\sigma}(c_i)
$
is a diffeomorphism onto a ball contained in $W$.
This proves \textup{(i)}.

\medskip
\noindent
\textit{Proof of item \textup{(ii)}.}

Fix $i\in\{1,\ldots,m\}$. For simplicity, write
\[
\exp:=\exp_{c_i},\qquad
u:=f\circ\exp,\qquad
\mathsf D:=B(0,4\sigma).
\]
Let $g_{jk}$ and $\Gamma^k_{jl}$ denote the metric
coefficients and Christoffel symbols in these normal
coordinates. 
The function \(u=f\circ\exp_{c_i}\) is locally Lipschitz on \(\mathsf D\). By Rademacher's theorem, \(u\) is differentiable almost everywhere, and its weak Euclidean gradient \(Du\) agrees almost everywhere with its classical gradient. Since \(f\) is locally \(L\)-Lipschitz with respect to the Riemannian distance, its weak Riemannian gradient satisfies
\[
|\nabla f|_g\le L
\quad\text{almost everywhere}.
\]
At almost every point \(y\in\mathsf D\), the chain rule gives
\[
Du(y)=d(f\circ\exp_{c_i})_y.
\]
Writing \(g(y)\) for the matrix of the metric in normal coordinates and identifying \(T_{c_i}M\) with \(\mathbb R^{n+1}\) via an orthonormal basis, we have
\[
|\nabla f|_g^2
=Du(y)^{T}g(y)^{-1}Du(y).
\]
The metric bounds in \eqref{eq:normal}, imply \(g(y)^{-1}\ge \frac12 I\). Consequently,
\[
\frac12|Du(y)|^2
\le Du(y)^Tg(y)^{-1}Du(y)
=|\nabla f|_g^2
\le L^2
\]
for almost every \(y\in\mathsf D\). Therefore,
\begin{equation}\label{eq:Du-bound}
|Du|\le\sqrt2\,L
\quad\text{almost everywhere in }\mathsf D.
\end{equation}

We first prove a distributional upper bound for the Euclidean Hessian of $u$. Fix $y_0\in\mathsf D$ and put
$x_0=\exp(y_0)$. Given $\varepsilon>0$, the support-sense assumption on $\operatorname{Hess}f$ provides a $C^2$
upper support function $\phi$ for $f$ at $x_0$ such that
\[
\phi\geq f,\qquad \phi(x_0)=f(x_0),
\qquad
\operatorname{Hess}\phi(x_0)
\leq(a+\varepsilon)g_{x_0}.
\]
Set $v:=\phi\circ\exp$. Then $v$ is an upper support
function for $u$ at $y_0$. By
Lemma~\ref{lem:supergrad},
\[
|\nabla\phi(x_0)|\leq L,
\]
and therefore
\[
|Dv(y_0)|\leq\sqrt2 L.
\]

The coordinate expression for the Riemannian Hessian is
\[
\operatorname{Hess}\phi(\partial_j,\partial_l)
=\partial_{jl}v-\Gamma^k_{jl}\partial_kv.
\]
Consequently, for every $\xi\in\mathbb R^{n+1}$,
\[
\begin{aligned}
D^2v(y_0)(\xi,\xi)
&=\operatorname{Hess}\phi(x_0)(\xi,\xi)
  +\Gamma^k_{jl}(y_0)\xi^j\xi^l\partial_kv(y_0)\\
&\leq(a+\varepsilon)g_{y_0}(\xi,\xi)
  +\|\Gamma(y_0)\|\,|\xi|^2|Dv(y_0)|\\
&\leq a\,g_{y_0}(\xi,\xi)
  +2\varepsilon|\xi|^2
  +4\sqrt2\,\Lambda L\sigma|\xi|^2.
\end{aligned}
\]
Here we used $g\leq2I$, $|y_0|<4\sigma$, and
$|Dv(y_0)|\leq\sqrt2L$.

Define the continuous symmetric matrix-valued function
\[
\mathsf B(y):=
a\,g(y)+4\sqrt2\,\Lambda L\sigma I.
\]
The preceding estimate implies that, for every
$\xi\in\mathbb R^{n+1}$ and every $\varepsilon>0$,
\[
D^2v(y_0)(\xi,\xi)
\leq \xi^{\mathsf T}\mathsf B(y_0)\xi
  +2\varepsilon|\xi|^2.
\]

We now apply Lemma~\ref{lem:1d} to the restrictions of
$u$ to straight lines. Fix a unit vector $\xi$ and write
\[
q(t):=u(z+t\xi),\qquad z\in\xi^\perp.
\]
For each point $t_0$ in the interval where the line
lies in $\mathsf D$, the preceding support inequality
shows that $q$ admits $C^2$ upper support functions whose
second derivatives at $t_0$ are bounded above by
\[
\beta_\xi(t_0)+\rho,
\qquad
\beta_\xi(t):=
\xi^{\mathsf T}\mathsf B(z+t\xi)\xi,
\]
for every $\rho>0$. Lemma~\ref{lem:1d} therefore yields
\[
\int_J q(t)\theta''(t)\,dt
\leq\int_J\beta_\xi(t)\theta(t)\,dt
\]
for every nonnegative $\theta\in C_c^\infty(J)$.

Integrating over the orthogonal complement $\xi^\perp$
and applying Fubini's theorem, we conclude that
\begin{equation}\label{eq:distributional-Hessian}
\int_{\mathsf D}u\,\partial_\xi^2\zeta\,dy
\leq
\int_{\mathsf D}
(\xi^{\mathsf T}\mathsf B(y)\xi)\zeta(y)\,dy
\end{equation}
for every nonnegative $\zeta\in C_c^\infty(\mathsf D)$.
By homogeneity, this inequality holds for every
$\xi\in\mathbb R^{n+1}$.

Let $\mu\in C_c^\infty(\mathbb R^{n+1})$ be a nonnegative
mollifier supported in the unit ball, with
$\int\mu=1$, and set
\[
\mu_\delta(y):=\delta^{-(n+1)}\mu(y/\delta),
\qquad
u_\delta:=u*\mu_\delta.
\]
For $0<\delta<\sigma$, the function $u_\delta$ is smooth
on $B(0,3\sigma)$.

Fix $y\in B(0,3\sigma)$. Applying
\eqref{eq:distributional-Hessian} to the test function
$z\mapsto\mu_\delta(y-z)$ gives
\[
\begin{aligned}
D^2u_\delta(y)(\xi,\xi)
&=\int_{\mathsf D}
u(z)\,\partial_\xi^2[\mu_\delta(y-z)]\,dz\\
&\leq\int_{\mathsf D}
(\xi^{\mathsf T}\mathsf B(z)\xi)
\mu_\delta(y-z)\,dz\\
&=\xi^{\mathsf T}(\mathsf B*\mu_\delta)(y)\xi.
\end{aligned}
\]

Since \(u\) is locally Lipschitz, it is Lipschitz on \(\overline{B(0,3\sigma)}\), with some Lipschitz constant \(C>0\). For \(y\in\overline{B(0,2\sigma)}\) and \(0<\delta<\sigma\), the support condition \(\operatorname{supp}(\mu_\delta)\subset\overline{B(0,\delta)}\) ensures that \(y-z\in B(0,3\sigma)\) whenever \(z\in\operatorname{supp}(\mu_\delta)\). Since \(\int_{\mathbb R^{n+1}}\mu_\delta(z)\,dz=1\), we obtain
\[
\begin{aligned}
|u_\delta(y)-u(y)|
&=\left|\int_{\mathbb R^{n+1}}
\bigl(u(y-z)-u(y)\bigr)\mu_\delta(z)\,dz\right|\\
&\le C\int_{\mathbb R^{n+1}}|z|\mu_\delta(z)\,dz\\
&=C\delta\int_{\mathbb R^{n+1}}|w|\mu(w)\,dw.
\end{aligned}
\]
The right-hand side is independent of \(y\) and tends to zero as \(\delta\to0\). Hence
\[
u_\delta\longrightarrow u
\quad\text{uniformly on }\overline{B(0,2\sigma)}.
\]

Since \(u\in W_{\mathrm{loc}}^{1,\infty}(\mathsf D)\), its weak gradient \(Du\) belongs to \(L_{\mathrm{loc}}^1(\mathsf D;\mathbb R^{n+1})\). For every compact subset of \(\mathsf D\) and all sufficiently small \(\delta>0\), differentiation of convolutions gives
\[
Du_\delta=(Du)*\mu_\delta.
\]
Consequently,
\[
Du_\delta(y)=
\int_{\mathbb R^{n+1}}Du(y-\delta z)\mu(z)\,dz.
\]
In particular, by \eqref{eq:Du-bound} and the
nonnegativity of $\mu_\delta$,
$
|Du_\delta(y)|
\leq\sqrt2 L.
$

By the Lebesgue differentiation theorem, almost every \(y\in B(0,2\sigma)\) is a Lebesgue point of \(Du\). Fix such a point \(y\). Since \(\mu\ge0\), \(\operatorname{supp}\mu\subset\overline{B(0,1)}\), and \(\mu\in L^\infty(\mathbb R^{n+1})\), we obtain
\[
\begin{aligned}
|Du_\delta(y)-Du(y)|
&\le \|\mu\|_{L^\infty}
\int_{B(0,1)}
|Du(y-\delta z)-Du(y)|\,dz\\
&=\|\mu\|_{L^\infty}\delta^{-(n+1)}
\int_{B(y,\delta)}
|Du(x)-Du(y)|\,dx\\
&=\|\mu\|_{L^\infty}|B(0,1)|
\frac{1}{|B(y,\delta)|}
\int_{B(y,\delta)}
|Du(x)-Du(y)|\,dx.
\end{aligned}
\]
The last expression tends to zero as \(\delta\to0^+\), by the defining property of a Lebesgue point. Hence
\[
Du_\delta(y)\longrightarrow Du(y)
\quad\text{for almost every }y\in B(0,2\sigma).
\]

Define
$
f_{i,\delta}:=u_\delta\circ\exp_{c_i}^{-1}$
on $B_{3\sigma}(c_i).$ Then $f_{i,\delta}$ is smooth. Its Riemannian gradient
satisfies
\[
|\nabla f_{i,\delta}|^2
=(Du_\delta)^{\mathsf T}g^{-1}Du_\delta
\leq2|Du_\delta|^2
\leq4L^2.
\]
Thus
\[
|\nabla f_{i,\delta}|\leq2L
\quad\text{on }B_{3\sigma}(c_i).
\]
The smooth dependence of the metric coefficients and
the almost-everywhere convergence of $Du_\delta$ imply
\[
\nabla f_{i,\delta}\longrightarrow\nabla f
\quad\text{almost everywhere on }U_i.
\]
Finally, the uniform convergence of $u_\delta$ gives
\[
f_{i,\delta}\longrightarrow f
\quad\text{uniformly on }\overline U_i.
\]
This proves \textup{(ii)}.

\medskip
\noindent
\textit{Proof of item \textup{(iii)}.}

We continue to use the normal coordinates centered at $c_i$.
From the distributional Hessian estimate established above,
\[
D^2u_\delta(y)(\xi,\xi)
\leq
a\,(g*\mu_\delta)_y(\xi,\xi)
+4\sqrt2\,\Lambda L\sigma|\xi|^2.
\]
On the other hand, the coordinate formula for the
Riemannian Hessian gives
\[
\begin{aligned}
\operatorname{Hess}f_{i,\delta}(\xi,\xi)
&=D^2u_\delta(y)(\xi,\xi)
-\Gamma^k_{jl}(y)\xi^j\xi^l\partial_k u_\delta(y)\\
&\leq a\,(g*\mu_\delta)_y(\xi,\xi)
+4\sqrt2\,\Lambda L\sigma|\xi|^2+\|\Gamma(y)\|\,|\xi|^2|Du_\delta(y)|.
\end{aligned}
\]
For $y\in B(0,2\sigma)$, we have
\[
\|\Gamma(y)\|\leq2\Lambda\sigma.
\]
Using $|Du_\delta|\leq\sqrt2L$, we obtain
\begin{equation}\label{eq:final-hessian-coordinate}
\operatorname{Hess}f_{i,\delta}(\xi,\xi)
\leq a\,(g*\mu_\delta)_y(\xi,\xi)
+6\sqrt2\,\Lambda L\sigma|\xi|^2.
\end{equation}

Suppose that $g_y(\xi,\xi)=1$. Since $g_y\geq\frac12 I$,
we have $|\xi|^2\leq2$. Consequently,
\[
\begin{aligned}
\operatorname{Hess}f_{i,\delta}(\xi,\xi)
&\leq a
+|a|\left|
((g*\mu_\delta)_y-g_y)(\xi,\xi)
\right|
+12\sqrt2\,\Lambda L\sigma\\
&\leq a
+2|a|\sup_{\overline{B(0,2\sigma)}}
\|g*\mu_\delta-g\|
+12\sqrt2\,\Lambda L\sigma.
\end{aligned}
\]
Here the first inequality follows by writing
\[
a(g*\mu_\delta)_y(\xi,\xi)
=a\,g_y(\xi,\xi)
+a\big((g*\mu_\delta)_y-g_y\big)(\xi,\xi)
\]
and using $g_y(\xi,\xi)=1$. Since the metric coefficients are uniformly continuous
on a compact neighborhood of $\overline{B(0,2\sigma)}$,
the approximate identity property implies that
\[
\sup_{\overline{B(0,2\sigma)}}
\|g*\mu_\delta-g\|
\longrightarrow0
\qquad\text{as }\delta\to0.
\]
Define
$
\omega_i(\delta):=
2|a|\sup_{\overline{B(0,2\sigma)}}
\|g*\mu_\delta-g\|.
$
Then $\omega_i(\delta)\to0$ as $\delta\to0$.
Moreover, our choice of $\sigma$ ensures that
$
12\sqrt2\,\Lambda L\sigma
\leq14\sqrt2\,\Lambda L\sigma
\leq\eta.
$
It follows that
\[
\operatorname{Hess}f_{i,\delta}(\xi,\xi)
\leq a+\eta+\omega_i(\delta)
\]
whenever $g_y(\xi,\xi)=1$. By homogeneity, we conclude that
\[
\operatorname{Hess}f_{i,\delta}
\leq
\bigl(a+\eta+\omega_i(\delta)\bigr)g
\]
on $U_i$. Since $f_{i,\delta}$ and the metric are smooth
on a neighborhood of $\overline U_i$, the inequality
extends to $\overline U_i$ by continuity. This proves
\textup{(iii)} and completes the proof.
\end{proof}

\bigskip

\noindent {\bf Acknowledgments:} During the preparation of this work, the authors used Claude, Gemini, and ChatGPT for refining arguments, locating references, and drafting and revising the manuscript. All AI-generated suggestions were independently verified by the authors, who take full responsibility for the content and accuracy of the final manuscript.

\end{document}